\documentclass[11pt]{amsart}

\usepackage{amssymb,graphicx}
\usepackage{verbatim}
\usepackage{color}

\usepackage[colorlinks=true,urlcolor=blue,
citecolor=red,linkcolor=blue,linktocpage,pdfpagelabels,bookmarksnumbered,bookmarksopen]{hyperref}

\newtheorem{teo}{Theorem}[section]
\newtheorem{lm}[teo]{Lemma}
\newtheorem{prop}[teo]{Proposition}

\newtheorem*{main}{Main Theorem}
\theoremstyle{definition}
\newtheorem{definition}[teo]{Definition}
\newtheorem*{nota}{Notation}
\newtheorem{oss}[teo]{Remark}
\newtheorem*{ack}{Acknowledgements}

\numberwithin{equation}{section}

\author[Bousquet]{Pierre Bousquet}
\address{Institut de Math\'ematiques de Toulouse, CNRS UMR 5219, Universit\'e de Toulouse, F-31062 Toulouse Cedex 9, France.}
\email{pierre.bousquet@math.univ-toulouse.fr}

\author[Brasco]{Lorenzo Brasco}
\address{Aix-Marseille Universit\'e, CNRS, Centrale Marseille, I2M, UMR 7373, 13453 Marseille, France}
\email{lorenzo.brasco@univ-amu.fr}

\title[Global Lipschitz continuity]{Global Lipschitz continuity\\ for minima of degenerate problems}

\date{\today}

\begin{document}

\begin{abstract}
We consider the problem of minimizing the Lagrangian $\int [F(\nabla u)+f\,u]$ among functions on $\Omega\subset\mathbb{R}^N$ with given boundary datum $\varphi$. We prove Lipschitz regularity up to the boundary for solutions of this problem, provided $\Omega$ is convex and $\varphi$ satisfies the bounded slope condition. The convex function $F$ is required to satisfy a qualified form of uniform convexity {\it only outside a ball} and no growth assumptions are made.
\end{abstract}

\maketitle

\begin{center}
\begin{minipage}{11cm}
\small
\tableofcontents
\end{minipage}
\end{center}

\section{Introduction}

\subsection{Aim of the paper}

Consider the following variational problem arising in the study of optimal thin torsion rods (see \cite{ABFL})
\[
\min_{u\in W^{1,2}_0(\Omega)}\int_{\Omega} \Big[F(\nabla u)-\lambda\,u\Big]\,dx.
\] 
Here \(\Omega\) is a bounded open set of \(\mathbb{R}^2\), \(\lambda\in \mathbb{R}^+\) and the function \(F\) is given by
\begin{equation}
\label{lagrangian_bouchitte_fragala}
F(z)= \left\{\begin{array}{lc}
|z|,& \mbox{ if } |z|\leq 1,\\
&\\
\displaystyle\frac{1}{2}\,|z|^2  + \frac{1}{2},& \mbox{ if } |z|>1.
\end{array} 
\right.
\end{equation}
By the direct methods in the Calculus of Variations, this problem admits at least a solution $u$. 
The regularity of \(u\) is obviously limited by the singularities and degeneracy of \(F\). 
In this respect, observe that \(F\) is non differentiable at \(0\) and the Hessian of \(F\) degenerates at any point of the unit ball.
\par
When \(\Omega\) is a disc centered at the origin, one can prove that this solution is unique, radial and Lipschitz continuous, but not \(C^1\). The main purpose of this paper is to establish that such a regularity property remains true in a more general framework. Notably, we merely assume that the domain \(\Omega\subset\mathbb{R}^N\) is convex (here \(N\geq 2\)) and we replace the parameter \(\lambda\) by a generic function \(f\in L^{\infty}(\Omega)\). We also allow for more general boundary conditions. Most importantly, we consider singular and degenerate Lagrangians, which may have a {\it wild growth at infinity}. Given such a convex function \(F\) which may be singular and/or degenerate inside a ball, we thus consider more generally the following problem: 
\begin{displaymath}
\tag{$P_0$}
\label{minimisation0}
\min\left\{\mathcal{F}(u):=\int_{\Omega} \Big[F(\nabla u) +f\,u\Big] \,dx\, :\, u-\varphi\in W^{1,1}_0(\Omega)\right\}.
\end{displaymath}
\subsection{A glimpse of BSC condition}
In order to neatly motivate the study of this paper and explain some of the difficulties we have to face, let us start by recalling some known facts about Lipschitz regularity.
One of the simplest instances of problem \eqref{minimisation0} is when \(f\equiv 0\) and $F$ is strictly convex. This substantially simplifies the situation since then:
\vskip.2cm
\begin{enumerate}
\item \label{property_1} a {\it comparison principle} holds true, i.e. if \(u\) and \(v\) are two solutions of \eqref{minimisation0} in \(\varphi+W^{1,1}_0(\Omega)\) and \(\psi+W^{1,1}_0(\Omega)\) respectively and \(\varphi\leq \psi\) on \(\partial \Omega\), then \(u\leq v\) on \(\Omega\) as well. This statement can be generalized to the case when \(F\) is merely convex but superlinear, see \cite{Ce, MT}.
\vskip.2cm
\item \label{property_2} an affine map \( v: x\mapsto \langle \zeta, x \rangle +a\) is a minimum.
\end{enumerate}
\vskip.2cm
The strict convexity of \(F\) also implies the uniqueness of the minimum. From the second observation above, it thus follows that when the boundary datum \(\varphi\) is affine, \(\varphi\) is the unique minimum. In particular, it is Lipschitz continuous. In contrast, when \(\varphi\) is assumed to be merely Lipschitz continuous, such a regularity property can not be deduced for a minimizer: even if \(F(\nabla u)=|\nabla u|^2\) and \(\Omega\) is a ball, the harmonic extension of a Lipschitz function \(\varphi:\partial \Omega \to \mathbb{R}\) is not Lipschitz in general (see \cite{Cl} for a counterexample), but only H\"older continuous (see e.g. \cite{BCVPDE,BMT2}).  
\par
These observations led to consider boundary data satisfying the so-called {\it bounded slope condition}. We  say that \(\varphi : \mathbb{R}^N \to \mathbb{R}\) satisfies the bounded slope condition (see also Section \ref{sec:2}) if \(\varphi|_{\partial \Omega}\) coincides with the restriction to \(\partial \Omega\) of a convex function \(\varphi^-\) and a concave function \(\varphi^+\). Equivalently,  \(\varphi|_{\partial\Omega}\) can be written as the supremum of a family of affine maps and also as the infimum of another family of affine maps.  When \(F\) is convex and \(f\equiv 0\), the bounded slope condition implies the existence of a Lipschitz solution to \eqref{minimisation0}, see \cite[Teorema 1.2]{Miranda} or \cite[Theorem 1.2]{Gi}.
The proof  relies in an essential way on the two properties \eqref{property_1} and \eqref{property_2}. 

\par
When \(f\not\equiv 0\), these two properties are false in general and the above approach must be supplemented with new ideas.  Stampacchia \cite{St} considered Lagrangians of the form \(F(\nabla u) + G(x,u)\) for some function \(G:\Omega\times \mathbb{R} \to \mathbb{R}\) satisfying suitable growth conditions (see also \cite{HaSt}). For bounded minima, there is no loss of generality in assuming  that \(G(x,u)\) has the form \(f(x)\,u\), see Section \ref{sec:5}.  In the following, we shall thus make this restriction. 
\par
In this case the bounded slope condition can still be exploited to obtain a Lipschitz regularity result when \(F\) is {\it uniformly convex}, in the following sense: \(F\) is \(C^2\) and there exists \(\mu>0\), \(\tau>-1/2\) such that for every \(z, \xi\in \mathbb{R}^N\),
\begin{equation}
\label{uniform_convexity_Stampacchia}
\langle D^2F(z)\,\xi, \xi\rangle\geq \mu\,(1+|z|^2)^{\tau}\,|\xi|^2.
\end{equation}
In this framework, if \(\Omega\) is uniformly convex  and \(f\in C^{0}(\overline{\Omega})\), then in \cite{St} it is proven that every solution of \eqref{minimisation0} is Lipschitz continuous. In Stampacchia's proof, the uniform convexity of \(F\) is used in a crucial way to compensate the pertubation caused by the lower order term \(f(x)\,u\). In this respect, we point out that when \(f\) is a constant map and \(F\) is {\it isotropic}\footnote{By this, we mean that \(F(z)=h(|z|)\) for some convex function \(h\).}, it is  possible to consider a more general class of functions \(F\), see \cite{FT}.

\subsection{Main results}
We now describe our contribution to the regularity theory for degenerate and singular Lagrangians which are {\it uniformly convex at infinity}. We first detail the uniform convexity property that will be considered in this paper. 
In what follows, we note by $B_R$ the $N-$dimensional open ball of radius $R>0$ centered at the origin.
\begin{definition}
\label{definition_uniformly_convex}
Let $\Phi:(0,+\infty)\to (0,+\infty)$ be a continuous function such that
\begin{equation}
\label{superlinear}
\lim_{t\to +\infty} t\, \Phi(t)=+\infty.
\end{equation}
We say that a map $F: \mathbb{R}^N \to \mathbb{R}$ is {\it $\Phi-$uniformly convex outside the ball} $B_R\subset \mathbb{R}^N$ if for every $z, z'\in \mathbb{R}^N$ such that the segment \([z, z']\) does not intersect \(B_R\) and for every  $\theta\in [0,1]$
\begin{equation}
\label{equc}
F(\theta\, z +(1-\theta)\,z')\leq \theta\, F(z) +(1-\theta)\,F(z')-\frac{1}{2}\,\theta\,(1-\theta)\,\Phi\big(|z|+|z'|\big)\,|z-z'|^2.
\end{equation}
If the previous property holds with $\Phi\equiv \mu>0$, we simply say that $F$ is {\it $\mu-$uniformly convex outside the ball} $B_R\subset \mathbb{R}^N$.
\end{definition}
The following is the main result of the paper:
\begin{main}
\label{teo:teoglobreg}
Let \(\Omega\subset\mathbb{R}^N\) be a bounded convex open set, \(\varphi:\mathbb{R}^N \to \mathbb{R}\) a Lipschitz continuous function, \(F:\mathbb{R}^{N}\to \mathbb{R}\) a convex function and $f\in L^\infty(\Omega)$. We consider the following problem
\begin{displaymath}
\tag{$P_\Phi$}
\label{minimisation}
\inf\left\{\mathcal{F}(u):=\int_{\Omega} \Big[F(\nabla u) +f\, u\Big] \,dx\, :\, u-\varphi\in W^{1,1}_0(\Omega)\right\}.
\end{displaymath}
Assume that \(\varphi|_{\partial \Omega}\) satisfies the bounded slope condition of rank $K\ge 0$ and that \(F\) is $\Phi-$uniformly convex outside the ball \(B_R\), for some $R>0$.
Then problem \eqref{minimisation} admits at least a solution and every such a solution is Lipschitz continuous. More precisely, we have
\[
\|u\|_{L^\infty(\Omega)}+\|\nabla u\|_{L^\infty(\Omega)}\le \mathcal{L}=\mathcal{L}(N,\Phi,K,R,\|f\|_{L^\infty(\Omega)},\mathrm{diam}(\Omega))>0,
\]
for every solution $u$.
\end{main}
Some comments on the assumptions of the previous result are in order.
\begin{oss}
Observe in particular that we allow for Lagrangians which are not necessarily \(C^1\). Moreover, {\it we do not assume any growth condition from above} on \(F\) at infinity. For example, the previous result covers the case of 
\begin{equation}
\label{exemples}
F(z)=\mu\,\left(|z|^p-\delta\right)_++\sum_{i=1}^N |z_i|^{p_i}\qquad \mbox{ or }\qquad F(z)=\mu\,\left(|z|-\delta\right)^p_++\sum_{i=1}^N |z_i|^{p_i},
\end{equation}
where $\mu>0$, $\delta\ge 0$, $p>1$ and $1<p_1\le \dots\le p_N$ without any further restriction. On the contrary, the case $\mu=0$ is not covered by our result, not even when $p_1=p_2=\dots=p_N$ (see the recent paper \cite{BBJ} for some results in this case).
Of course, many more general functions $F$ can be considered, not necessarily of power-type: for example
\[
F(z)=(|z|-\delta)_+\,\big[\log(1+|z|)\big]^p,
\]
with $\delta\ge 0$ and $p>1$ fulfills our hypothesis. The case $p=1$ is ruled out by condition \eqref{superlinear}.
%\par
%{\color{red}Concerning the lower order term, our assumption \(f\in L^{r}\) with \(r>N\) is sharp in the scale of Lebesgue spaces. Indeed, even when \(F(\nabla u)=|\nabla u|^2\) and $\varphi=0$, there exists \(f\in L^{N}(\Omega)\) such that the corresponding minimum \(u\) of \eqref{minimisation} is \emph{not} Lipschitz continuous.}
\par
Finally, the domain \(\Omega\) is convex (this is implicitly implied by the bounded slope condition), but not necessarily uniformly convex nor smooth.
\end{oss}

\subsection{Steps of the proof}

The first step of the proof is an approximation lemma which is new in several respects.
Given a bounded open convex set \(\Omega\) and a  function \(\varphi:\mathbb{R}^N\to \mathbb{R}\) which satisfies the bounded slope condition  of rank \(K\) for some \(K>0\), we construct 
\begin{itemize}
\item a sequence of smooth bounded open convex sets \(\Omega_k\supset\Omega\) converging to $\Omega$ (for the Hausdorff metric);
\vskip.2cm
\item a sequence of smooth functions \(\varphi_k\) which satisfy the bounded slope condition of rank \(K+1\) on \(\Omega_k\), such that \(\varphi_k=\varphi\) on \(\partial \Omega\).
\end{itemize} 
This construction relies on some properties of the  bounded slope condition that were initially discovered by Hartman \cite{Hartman-66, Hartman-68}.
In addition, we approximate the function \(F\) satisfying the $\Phi-$uniform convexity assumption \eqref{equc_bis} by a sequence of smooth  functions which are uniformly convex on the whole \(\mathbb{R}^N\). We are thus reduced to consider a variational problem $(P_k)$ for which the existence of a smooth solution $u_k$ is well-known. The goal is to establish regularity estimates on $u_k$, which are independent of \(k\).
\par
The second step is the construction of suitable barriers for the regularized problem $(P_k)$. Here the bounded slope condition plays a key role. This approach is quite standard but we have to overcome a new difficulty with respect to \cite{St}:  the stronger degeneracy of the Lagrangian. This is handled by introducing new explicit barriers adapted to this setting. From this construction, we deduce a uniform bound on \[
\|u_k\|_{L^{\infty}(\Omega_k)}+\|\nabla u_k\|_{L^{\infty}(\partial \Omega_k)}.
\]
In the third step of the proof, we obtain an estimate on the Lipschitz constant of \(u_k\). The method that we follow is classical in the setting of nonsmooth Lagrangians, which do not admit an Euler-Lagrange equation. We  compare a minimum \(u\) with its translations, namely functions of the form \(u(\cdot+\tau)\), \(\tau \in \mathbb{R}^N\). Once again, we have to cope with the degeneracy of the higher order part \(F\) of the Lagrangian, in presence of a term depending on \(x\) and \(u\). The main idea in \cite{St} was that the uniform convexity of \(F\) could be used to neutralize the lower order term \(f(x)\,u\). In our situation, this is only possible when \(|\nabla u|>R\). On the set where \(|\nabla u|\leq R\) instead, the gradient is obviously bounded almost everywhere, but this does not imply that the function \(u\) is Lipschitz continuous there, since we have no information on the regularity of this set. 
\par
Once a Lipschitz estimate independent of \(k\) is established, it remains to pass to the limit when \(k\) goes to \(\infty\) in order to establish the existence of a Lipschitz solution to the original problem \eqref{minimisation}. Since $\mathcal{F}$ is not strictly convex in general, this is not sufficient to infer that {\it every minimizer} is a Lipschitz function. In order to conclude the proof, we use that the lack of strict convexity is ``confined'' in the ball $B_R$, thus the Lipschitz regularity of a minimizer can be ``propagated'' to all the others (see Lemma \ref{propagationofregularity}).

\subsection{Comparison with previous results}
In order to handle the Lagrangian \(F\) given by \eqref{lagrangian_bouchitte_fragala}, an entirely different approach could have been followed.
Indeed, such an \(F\) has a {\it Laplacian structure at infinity}, i.e. for every $\xi\in\mathbb{R}^N$
\[
\langle D^2 F(z)\,\xi,\xi\rangle\simeq |\xi|^2,\qquad \mbox{ for }|z|\gg 1.
\]
Instead of exploiting the properties of the boundary condition \(\varphi\), one can rely on the specific growth property satisfied by \(F\) and prove a Lipschitz estimate by using test functions arguments and Caccioppoli-type inequalities.
\par
It is impossible to give a detailed account of all the contributions to the regularity theory of (local) minimizers of Lagrangians having more generally a \(p-\)Laplacian structure at infinity, i.e. such that for every $\xi\in\mathbb{R}^N$
\[
\langle D^2 F(z)\,\xi,\xi\rangle\simeq |z|^{p-2}\,|\xi|^2,\qquad \mbox{ for }|z|\gg 1.
\]
We cite the pioneering papers \cite{CE, GM, Ra}. More recently, many studies have been devoted to this subject, see for example \cite{Br, CF, CarM, DSV, DK, EMT, LPV} and \cite{SV}. 
Among the many contributions on the topic, we wish to mention the paper \cite[Theorem 2.7]{FFM} by Fonseca, Fusco and Marcellini. Here the Lagrangian has the form \(F(x,\nabla u)\) and is assumed to be \(p-\)uniformly convex at infinity in the following sense: there exist \(p>1\), \(\mu>0\) and \(R>0\) such that for every \(\xi, \xi'\in \mathbb{R}^N\setminus B_R\) and for every \(\theta \in [0,1]\),
\begin{equation}
\label{uniform_convexity}
F(x,\theta\, z+(1-\theta)\,z')\leq \theta\, F(x,z) + (1-\theta)\,F(x,z') -\frac{\mu}{2}\,\theta\,(1-\theta)\,(|z| + |z'|)^{p-2}|\xi-\xi'|^2.
\end{equation}
Observe that when \(F\) is \(C^2(\mathbb{R}^N)\), condition \eqref{uniform_convexity} coincides with \eqref{uniform_convexity_Stampacchia} for every \(z\in \mathbb{R}^N\setminus B_R\).
In addition, \(F\) is assumed to have $p-$growth, i.e.
\begin{equation}
\label{growth}
0\le F(x,\xi)\le L\,(1+|\xi|)^p,\qquad \xi\in\mathbb{R}^N.
\end{equation}
Then \cite[Theorem 2.7]{FFM} shows that every local minimizer is locally Lipschitz continuous. Observe that this holds for local minimizers, thus no regular boundary conditions $\varphi$ are needed. In particular, this kind of result can not be deduced from our Main Theorem.
\par
On the other hand, such a result is less general for two reasons: first of all, condition \eqref{uniform_convexity} is more restrictive than our \eqref{equc}, since it corresponds to the particular case $\Phi(t)=\mu\,t^{p-2}$; more importantly,
{\it we do not assume any growth condition} of the type \eqref{growth} on $F$.
\par
It should be pointed out that the technique of \cite{FFM} can be pushed further, by weakening \eqref{growth} and replacing it by a $q-$growth assumption, i.e.
\[
0\le F(x,\xi)\le L\,(1+|\xi|)^q,\qquad \xi\in\mathbb{R}^N,
\]
with $q>p$, see for example \cite[Theorem 1.1]{CGM}.  But still in this case, some restrictions are necessary: namely a condition like $q\le (1+C_N)\,p$ is required. Indeed, for $q$ and $p$ too far away, well-known examples show that local minimizers could be even unbounded (see \cite{Ho} for a counterexample and \cite{CoM,EMT,Ma89} for some regularity results on so-called {\it $(p,q)$ growth problems}). In particular, with these methods it is not possible to consider Lagrangians like \eqref{exemples}. 
As for global regularity, though not explicitely stated in \cite{FFM} or \cite{CGM}, we point out that the results of \cite{CGM,FFM} can be extended (as done in \cite{Br} for the Neumann case) to Lagrangians of the form \(F(\nabla u) + f(x)\,u\), provided $\Omega$ and the boundary datum are smooth enough.

\subsection{Plan of the paper}
In Section \ref{sec:2} we introduce the required notation and definitions. We recall some basic regularity results that will be needed throughout the whole paper. 
We also establish a new approximation lemma for a function \(\varphi\) which satisfies the bounded slope condition (this is Lemma \ref{lemma_approximation_sets}).
Then in Section \ref{sec:3} for the sake of completeness we show that problem \eqref{minimisation} admits solutions.
The proof of the Main Theorem is then contained in Sections \ref{sec:4} \& \ref{sec:5} : at first we show the result under the stronger assumption that $F$ is $\mu-$uniformly convex everywhere; then we deduce the general result by an approximation argument. Finally, Section \ref{sec:6} considers the case of more general functionals, where the lower order term $f(x)\,u$ is replaced by terms of the form $G(x,u)$. A (long) Appendix containing some results on uniformly convex functions complements the paper.

\begin{ack}
We warmly thank Guido De Philippis for pointing out a flaw in a preliminary version of this paper. Part of this work has been written during a visit of the first author to Marseille and of the second author to Toulouse. The IMT and I2M institutions and their facilities are kindly acknowledged. 
\end{ack}

\section{Preliminaries}
\label{sec:2}

\subsection{The bounded slope condition and approximation of convex sets}

\begin{definition}
\label{definition_bsc}
Let \(\Omega\) be a bounded open set in \(\mathbb{R}^N\) and $K>0$.
We say that a map \(\varphi : \partial \Omega \to \mathbb{R}\) satisfies the {\it bounded slope condition of rank $K$} if for every \(y\in \partial \Omega\), there exist \(\zeta_{y}^-, \zeta_{y}^+ \in \mathbb{R}^N\) such that \(|\zeta_{y}^-|, |\zeta_{y}^+ |\leq K\) and 
\begin{equation}
\label{eqBSC}
\varphi(y)+\langle \zeta_{y}^-, x-y\rangle \leq \varphi(x) \leq \varphi(y)+\langle \zeta_{y}^+, x-y\rangle,\qquad \mbox{ for every } x\in \partial \Omega.
\end{equation}
\end{definition}

\begin{oss}
We recall that whenever there exists a non affine function \(\varphi : \partial \Omega \to \mathbb{R}\)  satisfying the bounded slope condition, the set \(\Omega\) is necessarily convex (see \cite[Chapter 1, Section 1.2]{Gi}).
\end{oss}
As observed by Miranda \cite{Miranda} and Hartman \cite{Hartman-66}, there is a close relationship between this condition and the regularity of \(\varphi\). For example, we recall the following result contained in \cite[Corollaries 4.2 \& 4.3]{Hartman-66}.
\begin{prop}[\cite{Hartman-66}]
\label{teo:hartman}
If \(\Omega\subset\mathbb{R}^N\) is a $C^{1,1}$ open bounded convex set and \(\varphi\) satisfies the bounded slope condition, then \(\varphi\) is \(C^{1,1}\). If in addition $\Omega$ is assumed to be uniformly convex, then the converse is true as well, i.e. if \(\varphi\) is \(C^{1,1}\), then it satisfies the bounded slope condition.
\end{prop}
In this section, we  indicate how one can approximate a convex set \(\Omega\) by a sequence of smooth convex sets while preserving the bounded slope condition of a boundary map \(\varphi : \partial \Omega \to \mathbb{R}\). We first introduce some notation: given an open and bounded convex set $\Omega\subset\mathbb{R}^N$, 
we introduce the {\it normal cone at a point $x\in\partial\Omega$} by
\[
\mathcal{N}_\Omega(x)=\left\{\xi\in\mathbb{R}^N\, :\, \sup_{y\in\Omega}\,\langle \xi,y-x\rangle\le 0\rangle\right\}.
\]
We also set
\[
\mathcal{S}_\Omega(x)=\left\{\xi\in \mathcal{N}_\Omega(x)\, :\, |\xi|=1\right\}.
\]
Finally, given $x_0\in\Omega$ we introduce the {\it gauge function of $\Omega$ centered at $x_0$} by 
\[
j_{x_0,\Omega}(x)=\inf\{\lambda>0\, :\, x-x_0\in\lambda\,(\Omega-x_0)\}.
\]
This is a convex positively $1-$homogeneous function such that $j_{x_0,\Omega}(x_0)=0$ and  $j_{x_0,\Omega}\equiv 1$ on $\partial\Omega$. Moreover, this is a globally Lipschitz function, with Lipschitz constant given by
\[
\frac{1}{\mathrm{dist}(x_0,\partial\Omega)}.
\]
\vskip.2cm\noindent
We present a characterization of functions satisfying the bounded slope condition.
\begin{lm}
\label{lemmabsc}
 Let \(\Omega\subset\mathbb{R}^N\) be an open bounded convex set.
If \(\varphi: \partial\Omega \to \mathbb{R}\) satisfies the bounded slope condition of rank \(K\), then there exist two maps \(\varphi^- : \mathbb{R}^N\to \mathbb{R}\) , \(\varphi^+ : \mathbb{R}^N\to \mathbb{R}\) such that:
\begin{enumerate}
\item \(\varphi^-\) and \(\varphi^+\) are $(K+1)-$Lipschitz continuous, \(\varphi^-\) is convex, \(\varphi^+\) is concave and 
\[
\varphi^-|_{\partial \Omega} = \varphi^+|_{\partial \Omega} =\varphi;
\]
\item  \(\Omega=\{x\in \mathbb{R}^N: \varphi^+(x)>\varphi^-(x)\}\) , \(\partial \Omega = \{x\in \mathbb{R}^N :  \varphi^+(x)=\varphi^-(x) \}\) and there exists \(\beta_\Omega>0\) depending on $\Omega$ only (see Remark \ref{oss:beta} below) such that for every \(s>0\), 
\begin{equation}
\label{eq1797}
\{x\in \mathbb{R}^N : \varphi^-(x) \leq \varphi^+(x)+s\}\subset\overline{\{x\in\mathbb{R}^N\, :\, \mathrm{dist\,}(x,\Omega)\le \beta_\Omega\,s\}}.
\end{equation}
\end{enumerate} 
Conversely,  if a  convex map \(\varphi^-:\mathbb{R}^N\to \mathbb{R}\)  agrees with a concave map \(\varphi^{+}:\mathbb{R}^N \to \mathbb{R}\) on the boundary of  \(\Omega\), then \(\varphi:= {\varphi^-}_{|\partial \Omega} = {\varphi^+}_{|\partial \Omega} \) satisfies the  bounded slope condition of rank \(K\), where \(K\) is a common  Lipschitz rank for \(\varphi^-\) and \(\varphi^+\) on \(\overline{\Omega}\).
\end{lm}
\begin{proof}
Assume first that \(\varphi: \partial\Omega \to \mathbb{R}\) satisfies the bounded slope condition of rank \(K\). 
For every \(y\in \partial \Omega\), there exist \(\zeta_{y}^-, \zeta_{y}^+\in\mathbb{R}^N\) as in Definition \ref{definition_bsc}. We then define
\[
\varphi^{-}_0(x)=\sup_{y\in \partial\Omega}\left[\varphi(y) + \langle\zeta_{y}^-, x-y \rangle\right ],\qquad \mbox{ and }\qquad
\varphi^{+}_0(x)=\inf_{y\in \partial\Omega}\left[\varphi(y) + \langle\zeta_{y}^+ , x-y \rangle\right].
\] 
Then \(\varphi^{-}_0\) and \(\varphi^{+}_0\) are $K-$Lipschitz continuous, the former is convex, the latter is concave and they agree with \(\varphi\) on \(\partial \Omega\). This implies that
\[
\varphi^+_0 \ge \varphi^-_0\ \mbox{ on }\Omega,\qquad \mbox{ and }\qquad \varphi^+_0 \le \varphi^-_0\ \mbox{ on }\mathbb{R}^N\setminus\Omega.
\]
If we now define
\[
\varphi^{-}(x)=\varphi^-_0(x)+\mathrm{dist}(x_0,\partial\Omega)\,\Big(j_{x_0,\Omega}(x)-1\Big),
\]
and
\[
\varphi^+(x)=\varphi^+_0(x)+\mathrm{dist}(x_0,\partial\Omega)\,\Big(1-j_{x_0,\Omega}(x)\Big),
\]
these two functions have the required properties, thanks to the properties of the gauge function (see Figure \ref{fig:cones} below).
\begin{figure}[h]
\includegraphics[scale=.3]{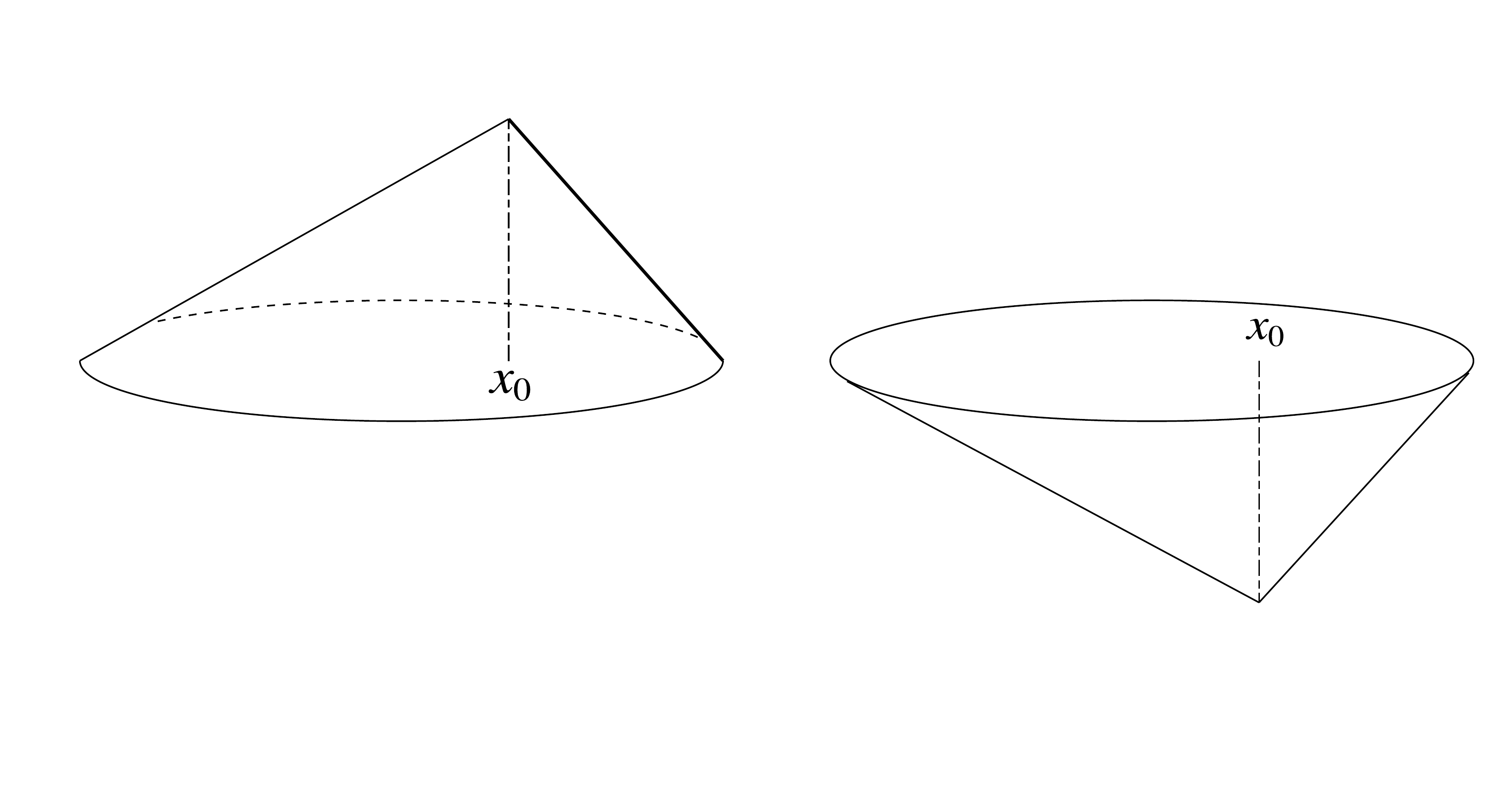}
\caption{The construction of the functions $\varphi^+$ and $\varphi^-$: we simply add the two cones to the functions $\varphi^+_0$ and $\varphi^-_0$, respectively.}
\label{fig:cones}
\end{figure}
\par
\vskip.2cm\noindent
We proceed to prove\footnote{We point out that the gauge functions are used  to guarantee property \eqref{eq1797}.} \eqref{eq1797}.  Let $x\in\mathbb{R}^N\setminus\Omega$ be such that
\[
\varphi^-(x)\le \varphi^+(x)+s,
\]
that is
\[
\varphi^-_0(x)+\mathrm{dist}(x_0,\partial\Omega)\,\Big(j_{x_0,\Omega}(x)-1\Big)\le \varphi^+_0(x)+\mathrm{dist}(x_0,\partial\Omega)\,\Big(1-j_{x_0,\Omega}(x)\Big)+s.
\]
Since on $\mathbb{R}^N\setminus\Omega$ we have $\varphi^-_0\ge \varphi^+_0$, then from the previous inequality we get
\[
\mathrm{dist}(x_0,\partial\Omega)\,\Big(j_{x_0,\Omega}(x)-1\Big)\le \mathrm{dist}(x_0,\partial\Omega)\,\Big(1-j_{x_0,\Omega}(x)\Big)+s,
\]
that is
\[
j_{x_0,\Omega}(x)\le 1+\frac{s}{2\,\mathrm{dist}(x_0,\partial \Omega)}.
\]
%This implies that
%\[
%x\in x_0+\left(1+\frac{M}{2\,\mathrm{dist}(x_0,\partial \Omega)}\right)\, (\Omega-x_0).
%\]
Then we observe
\[
\begin{split}
\mathrm{dist}(x,\partial\Omega)\le \left|\frac{x-x_0}{j_{x_0,\Omega}(x)}-(x-x_0)\right|&=\frac{|x-x_0|}{j_{x_0,\Omega}(x)}\, |1-j_{x_0,\Omega}(x)|\\
&\le \left(\max_{y\in\partial\Omega} |x_0-y|\right)\, \frac{s}{2\,\mathrm{dist}(x_0,\partial \Omega)}=:\beta_\Omega\, s,
\end{split}
\]
which gives \eqref{eq1797}.
\vskip.2cm\noindent
We now prove the converse. Assume that there exist a convex function \(\varphi^-:\mathbb{R}^N\to \mathbb{R}\) and a concave function \(\varphi^+:\mathbb{R}^N\to \mathbb{R}\) agreeing on \(\partial \Omega\) and let \(\varphi:=\varphi^-|_{\partial \Omega}=\varphi^+|_{\partial \Omega}\). For every \(y \in \partial \Omega\), let \(\zeta_{y}^-\in \partial \varphi^-(y)\) and  \(\zeta_{y}^+\in -\partial(-\varphi^+)(y)\). Then  for every \(x\in \partial \Omega\),
\[
\varphi(y) + \langle  \zeta_{y}^-, x-y \rangle \leq \varphi(x) \leq \varphi(y) + \langle \zeta_{y}^+, x-y \rangle.
\]
Moreover, if $\varphi^-$ and $\varphi^+$ are $K-$Lipschitz on \(\overline{\Omega}\), then \( |\zeta_{y}^-|\le K\)  and \( |\zeta_{y}^+|\le K\).
\end{proof}
\begin{oss}
\label{oss:beta}
In view of the above proof, in the previous result, one can take 
\[
\beta_\Omega=\frac{1}{2}\,\frac{\max_{y\in\partial \Omega} |x_0-y|}{\min_{y\in\partial\Omega}|x_0-y|}.
\]
By suitably choosing the point $x_0\in\Omega$, we can then suppose that $\beta_\Omega$ depends on $\Omega$ only through the ratio between its diameter and inradius (this quantity is sometimes called {\it eccentricity}).
\end{oss}
We proceed to describe the approximation of a bounded convex set by a sequence of smooth bounded convex sets that we will use in the sequel. Actually, the approximating sets can be chosen {\it uniformly} convex. 
\begin{lm}
\label{lemma_approximation_sets}
Let \(\Omega\subset\mathbb{R}^N\) be a convex bounded open set and let \(\varphi:\partial\Omega \to \mathbb{R}\) satisfy the bounded slope condition of rank \(K\). 
\par
Then there exists a sequence \(\{\Omega_k\}_{k\in\mathbb{N}}\subset \mathbb{R}^N\) of smooth bounded uniformly convex open sets, a sequence of $(K+1)-$Lipschitz functions \(\varphi_k:\mathbb{R}^N \to \mathbb{R}\) such that
\begin{enumerate}
\item[{\it i)}]  \(\Omega\Subset\Omega_k\) for every $k\in\mathbb{N}$ and we have the following
\[
\mathrm{diam}\,\Omega_k\leq \mathrm{diam}\,\Omega +\frac{8\,\beta_\Omega}{k},\qquad\qquad \lim_{k\to +\infty} |\Omega_k\setminus \Omega|=0,
\] 
and
\begin{equation}
\label{eq1728}
\lim_{k\to +\infty} \left[\max_{y\in\partial\Omega}\min_{y'\in \partial \Omega_k}|y- y'|\right]=0;
\end{equation}
\item[{\it ii)}]  \(\varphi_k|_{\partial \Omega_k}\) is smooth and  satisfies the bounded slope condition of rank \(K+2\);
\vskip.2cm
\item[{\it iii)}]  \(\varphi_k\) coincides with \(\varphi\) on \(\partial \Omega\). 
\end{enumerate}
 \end{lm}
\begin{proof}
Let \(\varphi^-, \varphi^+:\mathbb{R}^N\to \mathbb{R}\) be the two functions given by Lemma \ref{lemmabsc}. We consider \(\{\rho_{\varepsilon}\}_{\varepsilon>0}\subset C^\infty_0(B_\varepsilon)\) a sequence of standard mollifiers.
We define
\[
\psi^{-}_k:= \varphi^{-}\ast\rho_{\varepsilon_k}-\frac{1}{k}, \quad \mbox{ and }\quad \psi^{+}_k:= \varphi^{+}\ast\rho_{\varepsilon_k}+\frac{1}{k},\qquad \mbox{ where }\ \varepsilon_k=\frac{1}{2\,(K+1)\,k}.
\]
Of course, the first function is convex, while the second is concave.
Since \(\varphi^-, \varphi^+\) are $(K+1)-$Lipschitz, for every \(x\in \mathbb{R}^N\),
\begin{equation}\label{eq1874}
\psi^{-}_k(x) +  \frac{1}{2k}\leq \varphi^-(x) \leq \psi^{-}_k(x) +\frac{3}{2k}\quad \mbox{ and } \quad  \psi^{+}_k(x) -\frac{3}{2k}\leq  \varphi^{+}(x)\leq \psi^{+}_k(x)-\frac{1}{2k}.
\end{equation}
Let \(x_0\in \Omega\).
By Sard Lemma, for every $k\in\mathbb{N}\setminus\{0\}$ there exists \(\alpha_k\in (0,1/k)\) such that the set
\[
\Omega_k:=\left\{x\in \mathbb{R}^N\,:\, \psi^{+}_k(x) +\alpha_k>  \psi_{k}^{-}(x)+\frac{1}{2k\,\mathrm{diam}(\Omega)^2}\, |x-x_0|^2\right\},
\] 
is smooth. This set is uniformly convex, since the function
\[
x\mapsto \psi_{k}^{-}(x)-\psi^{+}_k(x)+\frac{1}{2k\,\mathrm{diam}(\Omega)^2}\, |x-x_0|^2,
\]
is uniformly convex. 
By \eqref{eq1874} and the fact that $\varphi^+\ge \varphi^-$ on $\overline\Omega$, we have for every $x\in\overline\Omega$
 \[
\psi^{-}_k(x)+\frac{1}{2k\,\mathrm{diam}(\Omega)^2}\, |x-x_0|^2\le \psi^+_k(x)-\frac{1}{k}+\frac{1}{2k\,\mathrm{diam}(\Omega)^2}\, |x-x_0|^2\le 
\psi^{+}_k (x),
\] 
and thus \(\overline{\Omega}\subset \Omega_k\).
Moreover, 
\begin{equation}
\label{eq1862}
\Omega_k\subset \left\{x\in \mathbb{R}^N\, :\, \varphi^-(x) < \varphi^+(x) + \frac{4}{k}\right\}.
\end{equation}
Hence, by using property \eqref{eq1797} of Lemma \ref{lemmabsc} with $s=4/k$ we have
\[
\mathrm{ diam }\, \Omega_k\leq \mathrm{ diam }\,\Omega + \frac{8\,\beta_\Omega}{k},
\] 
where $\beta_\Omega$ is the same constant as before.
We now estimate \(|\Omega_k\setminus \Omega|\). 
By \eqref{eq1862},
\[
\Omega_k\setminus \Omega \subset  \left\{x\in \mathbb{R}^N\, :\, \varphi^+(x)\leq \varphi^-(x) < \varphi^+(x) + \frac{4}{k}\right\}.
\]
It follows that
\[
\limsup_{k\to +\infty} |\Omega_k\setminus \Omega|\leq |\{x\in \mathbb{R}^N : \varphi^+(x)=\varphi^-(x)\}| = |\partial \Omega|,
\]
and the last quantity is $0$, by convexity of $\Omega$. Hence, \(\lim_{k\to +\infty}|\Omega_k\setminus \Omega|=0 \) as desired.
\vskip.2cm\noindent
In order to prove \eqref{eq1728}, let us argue by contradiction. Then there exist $r>0$, a subsequence $\{k_n\}_{n\in\mathbb{N}}$ diverging to $\infty$ and $y_n\in\partial\Omega$ such that for every $n\in\mathbb{N}$
\[
\min_{y' \in \partial \Omega_{k_n}}|y_n-y'|\ge r,
\]
that is, such that for every $n\in\mathbb{N}$
\[
|y_n-y'|\ge r,\qquad \mbox{ for every } y'\in\partial\Omega_{k_n}.
\]
The previous implies that $B_r(y_n)\subset \Omega_{k_n}$ for every $n\in\mathbb{N}$, thus we have\footnote{We denote by $\omega_N$ the volume of the $N-$dimensional ball of radius $1$.}
\[
|\Omega_{n_k}\setminus\Omega|\ge |B_r(y_n)\setminus \Omega|\ge \frac{\omega_N}{2}\,r^N.
\]
where in the last estimate we used the convexity of $\Omega$. 
The previous estimate clearly contradicts the fact that $|\Omega_k\setminus \Omega|$ converges to $0$.
\vskip.2cm\noindent
We now define \(\varphi_k\) by
\[
\varphi_k=\max \{\varphi^+,\, \psi^{-}_k\}.
\]
Of course, \(\varphi_k\) is $(K+1)-$Lipschitz on \(\mathbb{R}^N\).
Since \(\varphi^+=\varphi=\varphi^-\geq \psi^{-}_k\) on \(\partial \Omega\), we have \(\varphi_k = \varphi\) on \(\partial \Omega\). For \(x\in\partial \Omega_k\), we use the fact that \[
\varphi^+(x)\leq \psi^{+}_k(x)-\frac{1}{2k}=\psi^{-}_k(x)-\alpha_k +\frac{1}{2k\,\mathrm{diam}(\Omega)^2}\,|x-x_0|^2 -\frac{1}{2k} < \psi^{-}_k(x),
\]
to get \(\varphi_k= \psi^{-}_k\) on $\partial\Omega_k$. In particular, \(\varphi_k|_{\partial \Omega_k}\) is smooth and satisfies  the bounded slope condition of rank \(K+2\) for \(k\) sufficiently large. Indeed, we have
\[
\varphi_k|_{\partial \Omega_k}=\psi^-_k|_{\partial\Omega_k}=\left(\psi^{+}_k+\alpha_k-\frac{1}{2k\,\mathrm{diam}(\Omega)^2}\,|x-x_0|^2\right)|_{\partial\Omega_k},
\]
thus can we use the second part of Lemma \ref{lemmabsc} with the last two functions. This completes the proof.
\end{proof}
\begin{oss}
Since $\Omega_k\supset\Omega$, it is not difficult to see that \eqref{eq1728} is equivalent to say that $\Omega_k$ is converging to $\Omega$ with respect to the {\it Hausdorff metric}. Indeed, we recall that for every $E_1,E_2\subset\mathbb{R}^N$ the latter is given by
\[
d_H(E_1,E_2)=\max\left\{\sup_{x\in E_1} \mathrm{dist\,}(x, E_2),\, \sup_{y\in E_2} \mathrm{dist}(y, E_1)\right\}.
\]
If $E_2\subset E_1$, the previous reduces to
\[
d_H(E_1,E_2)=\sup_{x\in E_1} \mathrm{dist\,}(x, E_2)=\sup_{x\in E_1} \inf_{x'\in E_2} |x-x'|.
\]
Moreover, if both $E_1$ and $E_2$ are bounded and convex, we can equivalently perform the $\sup/\inf$ on their respective boundaries. 
\end{oss}

\subsection{Some auxiliary regularity results}

The following standard result will be used in the sequel.
\begin{teo}
\label{teo:stampacchia}
Let \(\mathcal{O} \subset \mathbb{R}^N\) be a smooth bounded uniformly convex open set, \(H:\mathbb{R}^N\to \mathbb{R}\) a smooth uniformly convex function and \(h:\mathcal{O}\to \mathbb{R}\) a smooth bounded function. Let \(\psi:\partial\mathcal{O} \to \mathbb{R}\) be a smooth map. Then there exists a unique minimizer \(u\) of 
\[
v\mapsto \int_{\mathcal{O}} \Big[H(\nabla v)+h\, v\Big]\, dx,
\] 
on \(\psi+W^{1,2}_0(\mathcal{O})\). Moreover, $u$ is smooth on \(\overline{\mathcal{O}}\).
\end{teo}
\begin{proof} 
We first observe that by Proposition \ref{teo:hartman}, the function $\psi$ satisfies the bounded slope condition. Then by \cite[Theorem 9.2]{St}, we have that the problem
\[
\min\left\{\int_{\mathcal{O}} H(\nabla v)\, dx+\int_{\mathcal{O}} h\, v\, dx\, :\, v-\psi\in W^{1,\infty}_0(\mathcal{O})\right\},
\]
admits a solution $u\in W^{1,\infty}(\mathcal{O})$, which is also unique by strict convexity of the Lagrangian in the gradient variable. Then $u$ satisfies the Euler-Lagrange equation
\begin{equation}
\label{EL_stampacchia}
\int_{\mathcal{O}} \langle \nabla H(\nabla u),\nabla \varphi\rangle\, dx+\int_{\mathcal{O}} h\, \varphi\, dx=0,\qquad \mbox{ for every } \varphi\in W^{1,\infty}_0(\mathcal{O}).
\end{equation}
Thanks to the hypotheses on $H$ and $h$ and to the fact that $\nabla u\in L^\infty(\mathcal{O})$, equation \eqref{EL_stampacchia} still holds with test functions $\varphi\in W^{1,1}_0(\mathcal{O})$. Thus by convexity, $u$ solves the original problem as well.
\par
By using \eqref{EL_stampacchia}, a standard difference quotient argument and the Lipschitz continuity of \(u\), one can prove that \(u\in W^{2,2}(\mathcal{O})\) and that any partial derivative \(u_{x_i}\) is a solution of a uniformly elliptic equation with bounded measurable coefficients, i.e.
\[
-\mathrm{div}\left(D^2 H(\nabla u) \nabla u_{x_i}\right)=h_{x_i}.
\] 
By the De Giorgi-Nash-Moser Theorem, it thus follows that \(u_{x_i}\in C^{0,\alpha}(\overline{\mathcal{O}})\). By appealing to Schauder theory, this implies that \(u\) is smooth on \(\overline{\mathcal{O}}\), see \cite[Theorem 9.19]{GT}. 
\end{proof}
Finally, in the proof of the Main Theorem we will also need the following simple result.
\begin{lm}[Propagation of regularity]
\label{propagationofregularity} Under the assumptions of the Main Theorem, if  \(u_1\) and \(u_2\) are two solutions of \eqref{minimisation}, then 
\[
|\nabla u_1-\nabla u_2|\le 2\,R,\qquad \mbox{ almost everywhere in }\Omega.
\]
\end{lm}
\begin{proof}
From the identity \(\mathcal{F}(u_1)=\mathcal{F}(u_2)\) and the convexity of \(\mathcal{F}\) we easily infer that 
\[
\mathcal{F}(u_1)=\mathcal{F}\left(\frac{u_1+u_2}{2}\right).
\]
Hence, for almost every \(x \in \Omega\), 
\[
F\left(\frac{\nabla u_1(x)+\nabla u_2(x)}{2}\right) = \frac{1}{2}\,F(\nabla u_1(x)) + \frac{1}{2}\, F(\nabla u_2(x)).
\]
Since \(F\) is strictly convex outside \(B_R\), this implies that either \(\nabla u_1(x)= \nabla u_2(x)\) or \(\nabla u_1(x)\) and \(\nabla u_2(x)\) both belong to \(B_R\). In any case, \(|\nabla u_1(x) - \nabla u_2(x)|\leq 2\,R\) for almost every $x\in\Omega$. This concludes the proof.
\end{proof}

\section{Existence of minimizers}
\label{sec:3}

In this section, we prove that the minimum in \eqref{minimisation} is attained, under the standing assumptions. We begin by remarking a consequence of $\Phi-$uniform convexity.
\begin{oss}
Observe that if $F$ is $\Phi-$uniformly convex outside $B_R$, then for every $\xi, \xi'\in \mathbb{R}^N$ such that the segment \([\xi, \xi']\) does not intersect \(B_R\) and every $\zeta\in\partial F(\xi)$, we also have
\begin{equation}
\label{equc_bis}
F(\xi')\ge F(\xi)+\langle \zeta,\xi'-\xi\rangle+\frac{1}{2}\,\Phi(|\xi|+|\xi'|)\, |\xi-\xi'|^2.
\end{equation}
Indeed, from \eqref{equc} we get for $0<\theta<1$
\[
F(\xi +\theta\,(\xi'-\xi))-F(\xi)\leq \theta\, \Big(F(\xi')-F(\xi)\Big)-\frac{1}{2}\,\theta\,(1-\theta)\,\Phi(|\xi|+|\xi'|)\,|\xi-\xi'|^2,
\]
then if $\zeta\in \partial F(\xi)$, by convexity of $F$ we have
\[
\theta\,\langle \zeta,\xi'-\xi\rangle\le \theta\, \Big(F(\xi')-F(\xi)\Big)-\frac{1}{2}\,\theta\,(1-\theta)\,\Phi(|\xi|+|\xi'|)\,|\xi-\xi'|^2.
\]
By dividing by $\theta$ and taking the limit as $\theta$ goes to $0$, we get \eqref{equc_bis}.
\end{oss}
We can now establish that $F$ is superlinear.
\begin{lm}
\label{lm:superlinear}
Let $F:\mathbb{R}^N\to\mathbb{R}$ be a continuous function, such that \(F\) is $\Phi-$uniformly convex outside the ball \(B_R\). Then $F$ is superlinear.
\end{lm}
\begin{proof}
We recall that by hypothesis $F$ satisfies \eqref{equc_bis}, i.e. for every $\xi,\xi'\in\mathbb{R}^N$ such that $[\xi,\xi']$ does not intersect $B_R$, we have
\[
F(\xi)\ge F(\xi')+\langle \zeta,\xi-\xi'\rangle+\frac{1}{2}\,\Phi(|\xi|+|\xi'|)\, |\xi-\xi'|^2,
\]
where $\zeta\in \partial F(\xi')$. Let $\xi\in\mathbb{R}^N\setminus B_{2\,R}$ and $\xi'=2\,R\,\xi/|\xi|$, by using the previous property, Cauchy-Schwarz inequality and continuity of $F$ we get
\[
\frac{F(\xi)}{|\xi|}\ge \frac{1}{|\xi|}\, \left[\min_{\omega\in \mathbb{S}^{N-1}}\, F(2\,R\,\omega)-|\zeta|\,(|\xi|-2\,R)\right]+\frac{1}{2}\,\Phi(|\xi|+2\,R)\, \frac{(|\xi|-2\,R)^2}{|\xi|},
\]
where the norm of $\zeta\in\partial F(2\,R\,\xi/|\xi|)$ can be bounded by \(\|\nabla F\|_{L^{\infty}(B_{2R})}\).
If we now use the assumption \eqref{superlinear} on $\Phi$, from the previous estimate we get that
\[
\lim_{|\xi|\to +\infty} \frac{F(\xi)}{|\xi|}=+\infty,
\]
which is the desired conclusion.
\end{proof}
We now prove an existence result for a problem having a slightly more general form than \eqref{minimisation}, namely we consider  functionals of the form 
\[
\mathcal{F}(u)=\int_{\Omega} \Big[F(\nabla u) + G(x,u)\Big] \,dx,
\]
where \(G:\Omega \times \mathbb{R} \to \mathbb{R}\) is a measurable function which satisfies the following assumption: there exist two functions $g_1\in L^N(\Omega)$ and $g_2\in L^1(\Omega)$ such that 
\begin{equation}
\label{Gip}
G(x,u)\ge -|u|\, |g_1(x)|-|g_2(x)|,\qquad \mbox{ for a.\,e. } x\in \Omega \mbox{ and every } u\in \mathbb{R}.  
\end{equation}
We further assume that there exists \(u_*\in \varphi + W^{1,1}_0(\Omega)\) such that
\begin{equation}
\label{Gip_bis}
\int_\Omega \Big|F(\nabla u_*) +G(x,u_*)\Big| \,dx<+\infty. 
\end{equation}
We have the following existence result.
\begin{prop}
\label{prop:existence}
Let \(\Omega\subset\mathbb{R}^N\) be a bounded convex open set and \(\varphi:\mathbb{R}^N \to \mathbb{R}\) a Lipschitz continuous function. Let \(F:\mathbb{R}^{N}\to \mathbb{R}\) be a convex function, which is  $\Phi-$uniformly convex outside the ball \(B_R\) and $G:\Omega\to\mathbb{R}$ a measurable function satisfying \eqref{Gip} and \eqref{Gip_bis}. Then the following problem
\begin{equation}
\label{minexi}
\inf\left\{\mathcal{F}(u):=\int_{\Omega} \Big[F(\nabla u) +G(x,u)\Big] \,dx\, :\, u-\varphi\in W^{1,1}_0(\Omega)\right\},
\end{equation}
admits a solution. 
\end{prop}
\begin{proof}
By Lemma \ref{lm:superlinear} we know that $F$ is superlinear. In particular, we get that for every $M >0$, there exists $r=r(M)>0$ such that
for every $\xi\in \mathbb{R}^N\setminus B_r$ 
\begin{equation}
\label{M}
F(\xi)\ge M\,|\xi|.
\end{equation}
For every $u\in W^{1,1}(\Omega)$ we thus get 
\[
F(\nabla u)\ge \min\Big(M\,|\nabla u|,\min_{B_r} F\Big).
\]
From \eqref{Gip}, it also follows that 
\[
G(x,u)\ge -|u|\,|g_1(x)|-|g_2(x)|.
\]
Hence, we get that $F(\nabla u)+G(x,u)$ is greater than or equal to an $L^1(\Omega)$ function. This proves that the functional $\mathcal{F}$ is well-defined on the class $\varphi+W^{1,1}_0(\Omega)$.
\vskip.2cm\noindent
In order to prove that \eqref{minimisation} admits a solution, we first observe that the functional is not constantly $+\infty$, since by \eqref{Gip_bis}, $\mathcal{F}(u_*)<+\infty$. By using H\"older and Sobolev inequalities we get
\[
\begin{split}
\int_\Omega |u|\,|g_1|\,dx&\le \|g_1\|_{L^N(\Omega)}\,\left(S_N\, \|\nabla u\|_{L^{1}(\Omega)}+S_N\,\|\nabla \varphi\|_{L^{1}(\Omega)}+\|\varphi\|_{L^{N'}(\Omega)}\right)\\
&\le C'_1\, \|g_1\|_{L^N(\Omega)}\, \|\nabla u\|_{L^{1}(\Omega)}+C_2',
\end{split}
\]
for some $C_1'=C_1'(N)>0$ and $C_2'=C_2'(N,\|g_1\|_{L^N(\Omega)},\|\varphi\|_{W^{1,1}(\Omega)})>0$. In conclusion, we obtain
\begin{equation}
\label{frombelow}
\begin{split}
\int_\Omega F(\nabla u)\,dx+\int_\Omega G(x,u)\,dx&\ge \int_\Omega F(\nabla u)\,dx
-C_1'\,\|g_1\|_{L^N(\Omega)}\, \int_\Omega |\nabla u|\,dx-C_2''.
\end{split}
\end{equation}
with $C_2''=C_2''(N,\|g_1\|_{L^N(\Omega)},\|g_2\|_{L^1(\Omega)},\|\varphi\|_{W^{1,1}(\Omega)})>0$. 
\par
We now take a minimizing sequence $\{u_n\}_{n\in\mathbb{N}}\subset \varphi+W^{1,1}_0(\Omega)$, thanks to the previous discussion we can suppose that
\[
\int_\Omega \Big[F(\nabla u_n)+G(x,u_n)\Big]\,dx\le C,\qquad \mbox{ for every }n\in\mathbb{N}.
\]
In particular, by using \eqref{frombelow}, for every $n\in\mathbb{N}$ we get
\begin{equation}
\label{intermediaire}
\int_\Omega F(\nabla u_n)\,dx \le C+C'_1\, \|g_1\|_{L^N(\Omega)}\, \int_\Omega |\nabla u_n|\,dx+C''_2.
\end{equation}
We now claim that the previous estimate implies
\begin{equation}
\label{borne}
\int_\Omega F(\nabla u_n)\,dx\le C,\qquad \mbox{ for every } n\in\mathbb{N},
\end{equation}
for a different constant $C=C(N,|\Omega|,\|g_1\|_{L^N(\Omega)},\|g_2\|_{L^1(\Omega)},\|\varphi\|_{W^{1,1}(\Omega)},F)>0$. Indeed, if $g_1\equiv 0$, then there is nothing to prove. Let us suppose $\|g_1\|_{L^N(\Omega)}>0$.
We now take
\[
M=2\,C_1'\,\|g_1\|_{L^N(\Omega)},
\] 
we get
\[
C'_1\, \|g_1\|_{L^N(\Omega)}\, \int_\Omega |\nabla u_n|\,dx=\frac{M}{2}\, \int_\Omega |\nabla u_n|\,dx\le \frac{1}{2}\,\int_\Omega F(\nabla u_n)\,dx+\frac{M}{2}\,r\,|\Omega|.
\]
By using this information into \eqref{intermediaire}, we get the claimed estimate \eqref{borne}.  Since $F$ is superlinear, estimate \eqref{borne} implies that $\{\nabla u_n\}_{n\in\mathbb{N}}$ is equi-integrable (\cite[Lemma 1.9.1]{Mor}). Then Dunford-Pettis Theorem implies that a subsequence of $\{\nabla u_n\}_{n\in\mathbb{N}}$ (we do not relabel) weakly converges in $L^1$ to $\phi\in L^1(\Omega;\mathbb{R}^N)$. By Rellich Theorem, we may also assume that the sequence $\{u_n\}_{n\in\mathbb{N}}$ is  strongly converging in $L^1(\Omega)$ to a function $u$. It is easy to see that $u\in W^{1,1}(\Omega)$ and $\nabla u=\phi$, since for every $\psi\in C^\infty_0(\Omega;\mathbb{R}^N)$ we have
\[
\int_\Omega u\,\mathrm{div\,}\psi\,dx=\lim_{n\to\infty}\int_\Omega u_n\,\mathrm{div\,}\psi\,dx=-\lim_{n\to\infty}\int_\Omega \langle \nabla u_n,\psi\rangle\,dx=-\int_\Omega \langle \phi,\psi\rangle\,dx.
\]
Then $u$ is admissible for the variational problem. By lower semicontinuity of the functional we get the desired result.
\end{proof}

\section{A weaker result: the case of $\mu-$uniformly convex problems} 
\label{sec:4}

In this section, as an intermediate result we prove the following weaker version of the Main Theorem. This will be used in the next section. The form of the Lipschitz estimate \eqref{lipschitz} below will play an important role.
\begin{teo}
\label{teo:minimisationk}
Let \(\Omega\subset\mathbb{R}^N\) be a bounded convex open set, \(\varphi:\mathbb{R}^N \to \mathbb{R}\) a Lipschitz continuous function, \(F:\mathbb{R}^{N}\to \mathbb{R}\) a convex function and $f\in L^\infty(\Omega)$. We consider the following problem
\begin{displaymath}
\tag{$P_\mu$}
\label{minimisation_weak}
\min\left\{\mathcal{F}(u):=\int_{\Omega} \Big[F(\nabla u) +f\, u\Big] \,dx\, :\, u-\varphi\in W^{1,1}_0(\Omega)\right\}.
\end{displaymath}
Assume that \(\varphi|_{\partial \Omega}\) satisfies the bounded slope condition of rank $K>0$ and that \(F\) is $\mu-$uniformly convex outside some ball \(B_R\) for some $0<\mu\le 1$.
\par
Then every solution \(u\) of \eqref{minimisation_weak} is Lipschitz continuous. 
More precisely, there exists a constant $\mathcal{L}_0=\mathcal{L}_0(N,K,R,\|f\|_{L^\infty(\Omega)},\mathrm{diam}(\Omega))>0$ such that
\begin{equation}
\label{lipschitz}
\|\nabla u\|_{L^\infty(\Omega)}\le \frac{\mathcal{L}_0}{\mu},
\end{equation}
for every solution $u$.
\end{teo}
In order to neatly present the proof of this result, we divide this section into subsections, each one corresponding to a step of the proof.

\subsection{Step 0: a regularized problem}
We will proceed by approximation. We consider a nondecreasing sequence of smooth convex functions
\(\{F_k\}_{k\in \mathbb{N}}\)  on \(\mathbb{R}^N\) which converges uniformly on bounded sets to \(F\) and such that for every $k\in\mathbb{N}$
\begin{equation}
\label{eqFkuc1}
\langle D^2 F_k (x)\,\eta,\eta\rangle \geq \frac{\mu}{36}\,|\eta|^2,\quad\qquad \mbox {for every } x\in \mathbb{R}^N\setminus B_{R+1} \mbox{ and every } \eta \in\mathbb{R}^N.
\end{equation} 
See Lemma \ref{lemma_Fk} in Appendix \ref{section_approximation_convex_functions}.
\vskip.2cm
We also introduce  a sequence \(\{f_k\}_{k\in \mathbb{N}}\subset C^{\infty}_0(\mathbb{R}^N)\) which converges $\ast-$weakly to \(f\) in \(L^\infty(\Omega)\) and such that for every $k\in\mathbb{N}$
\begin{equation}
\label{norme}
\|f_k\|_{L^{\infty}(\mathbb{R}^N)}\le (\|f\|_{L^\infty(\Omega)}+1)=:\Lambda.
\end{equation}
Finally, let \(\{\Omega_k\}_{k\in\mathbb{N}}\) and \(\{\varphi_k\}_{k\in\mathbb{N}}\) be as in Lemma \ref{lemma_approximation_sets}.
We then define
\[
\mathcal{F}_k(v)=\int_{\Omega_k} F_k(\nabla v) \,dx+\frac{1}{k}\,\int_{\Omega_k} |\nabla v|^2\,dx + \int_{\Omega_k} f_k\,v\,dx,
\]
and consider the following problem 
\begin{displaymath}
\tag{$P_{\mu,k}$}
\label{minimisationk}
\min\left\{\mathcal{F}_k(u)\, :\, u-\varphi_k\in W^{1,1}_0(\Omega_k)\right\}.
\end{displaymath}
Existence of a solution to \eqref{minimisationk} follows from Proposition \ref{prop:existence}. Moreover, since the Lagrangian is strictly convex in the gradient variable, such a solution is unique. We will denote it by \(u_k\in\varphi_k +W^{1,1}_0(\Omega_k)\). Observe that by Theorem \ref{teo:stampacchia}, we know that \(u_k\) is smooth on \(\overline{\Omega}_k\). We aim at proving a global Lipschitz estimate independent of $k$.
\begin{nota}
In what follows, in order to simplify the notation, we denote the function \(u_k\) by \(U\) and the set \(\Omega_k\) by \(\mathcal{O}\).
\end{nota}

\subsection{Step 1: construction of uniform barriers}

Without loss of generality, we can assume that \(K+2\) (the rank given by the bounded slope condition, see Lemma \ref{lemma_approximation_sets}) is also the Lipschitz constant of \(\varphi_k\) on the whole \(\mathbb{R}^N\) (we only need to redefine \(\varphi_k\) outside \(\partial\mathcal{O}\) if necessary).
\vskip.2cm
The proof of Theorem \ref{teo:minimisationk} is based on the method of barriers that we now construct explicitly. In passing, we will also prove that minimizers of \eqref{minimisationk} are bounded, uniformly in $k$.
\begin{prop}
\label{lm:barrier}
With the previous notation, there exist a constant
\[
L_0=L_0\big(N,K, R,\|f\|_{L^\infty(\Omega)},\mathrm{diam}(\Omega)\big)>0,
\]
and two $(L_0/\mu)-$Lipschitz maps \(\ell^-, \ell^+ : \mathbb{R}^N \to \mathbb{R}\) with the following properties:
\begin{itemize}
\item[(i)] for the solution \(U\) of \eqref{minimisationk}, for every \(x\in \mathcal{O}\),
\[
\ell^-(x)\leq U(x) \leq \ell^+(x),
\]
so that in particular $U\in L^\infty(\mathcal{O})$, with an estimate independent of $k$;
\vskip.2cm
\item[(ii)] \(\ell^-=\ell^+=\varphi_k\) on \(\mathbb{R}^N\setminus\mathcal{O}\).
\end{itemize}
\end{prop}
\begin{proof} 
We only give the construction for $\ell^-$, since the one for $\ell^+$ is completely analogous.
Let $y\in\partial\mathcal{O}$, we then take \(\nu_{\mathcal{O}}(y)\) the unit outer normal to \(\mathcal{O}\) at \(y\). Recall that by convexity  of $\mathcal{O}$ 
\begin{equation}
\label{diametro}
\langle \nu_\mathcal{O}(y), x-y \rangle \leq 0,\qquad \mbox{ for every }x\in\mathcal{O}.
\end{equation}
We introduce the function $a_{k,y}$ defined by
\[
a_{k,y}(x)=\Big[\left(2\,\mathrm{diam}(\mathcal{O})+\langle \nu_{\mathcal{O}}(y),x-y\rangle\right)^2-4\, \mathrm{diam}(\mathcal{O})^2\Big].
\]
Observe that this is a convex function such that
\begin{equation}
\label{maximum}
\Delta a_{k,y}=2,\ \mbox{ in }\mathcal{O},\qquad a_{k,y}\le 0 \mbox{ in }\overline{\mathcal{O}} \qquad \mbox{ and }\qquad a_{k,y}(y)=0.
\end{equation}
We also observe that by Lemma \ref{lemma_approximation_sets}, up to choosing $k$ sufficiently large, we can suppose that
\begin{equation}
\label{diametri}
\mathrm{diam}(\Omega)\le\mathrm{diam}(\mathcal{O})\le \mathrm{diam}(\Omega)+1.
\end{equation}
Finally, we define for \(x\in \mathbb{R}^N\),
\[
\psi_{y}^-(x)=\varphi_k(y) + \langle \zeta_{y}^-, x-y\rangle +T\,a_{k,y}(x),
\]
where \(\zeta_{y}^-\) is chosen as in \eqref{eqBSC} and
\[
T:=\frac{18}{\mu}\,\left(\frac{R+K+3}{\mathrm{diam}(\Omega)}+\Lambda+1\right).
\]
Recall that $\Lambda$ is the constant defined by \eqref{norme} and that $0<\mu\le 1$, by hypothesis.   
By construction we have \(\psi_{y}^-(y)=\varphi_k(y)\) and \(\psi_{y}^-\) is $(L_0/\mu)-$Lipschitz on \(\mathcal{O}\), with 
\begin{equation}
\label{L_0}
L_0=18\,\left[4\,\left(\frac{R+K+3}{\mathrm{diam}(\Omega)}+\Lambda+1\right)\,(\mathrm{diam}(\Omega)+1)+K+2\right].
\end{equation} 
In order to compute $L_0$, we also used that
\begin{equation}
\label{2fois}
\mathrm{diam}(\Omega)\le |\nabla a_{k,y}(x)|=2\,\Big|2\,\mathrm{diam}(\mathcal{O})+\langle \nu_{\mathcal{O}}(y),x-y\rangle\Big|\le 4\,\mathrm{diam}(\Omega)+4,
\end{equation}
which follows from the convexity of $\mathcal{O}$ and \eqref{diametri}.
Moreover, by \eqref{eqBSC} and \eqref{maximum}, we have
\[
\psi_y^-\le\varphi_k,\qquad \mbox{ on }\partial\mathcal{O}.
\]
Let \(U\) be the minimum of \eqref{minimisationk}. By testing the minimality of $U$ against the function $\max\{U, \psi_{y}^-\}$ we get
\begin{equation}
\label{eq180}
\begin{split}
\int_{\{U<\psi_{y}^-\}} \big[F_k(\nabla U) -F_k(\nabla \psi_{y}^-)\big ]\, dx&+\frac{1}{k}\, \int_{\{U<\psi_y^-\}} \Big[|\nabla U|^2-|\nabla \psi^-_y|^2\Big]\,dx\\
&\leq \int_{\{U<\psi_{y}^-\}} f_k\,(\psi_{y}^--U)\, dx\leq \int_{\{U<\psi_{y}^-\}} f^+_k\,(\psi_{y}^--U)\, dx.
\end{split}
\end{equation}
By using convexity in the left-hand side of \eqref{eq180}, we can estimate
\begin{equation}
\label{passagelimite}
\begin{split}
\int_{\{U<\psi_{y}^-\}} \big[F_k(\nabla U) -F_k(\nabla \psi_{y}^-)\big]\, dx &+\frac{1}{k}\, \int_{\{U<\psi_y^-\}} \Big[|\nabla U|^2-|\nabla \psi^-_y|^2\Big]\,dx\\
& \geq \int_{\{U<\psi_{y}^-\}} \langle \nabla F_k(\nabla \psi_{y}^-) , \nabla U-\nabla \psi_{y}^- \rangle\, dx\\
&+\frac{2}{k}\,\int_{\{U<\psi^-_y\}} \langle \nabla \psi_y^-,\nabla U-\nabla \psi^-_y\rangle\,dx.
\end{split}
\end{equation}
By integration by parts, 
\[ 
\begin{split}
\int_{\{U<\psi_{y}^-\}} \langle \nabla F_k(\nabla \psi_{y}^-) , \nabla U-\nabla \psi_{y}^- \rangle \, dx&+\frac{2}{k}\,\int_{\{U<\psi_{y}^-\}} \langle \nabla \psi_{y}^- , \nabla U- \nabla \psi^-_{y} \rangle\, dx\\
&=  \int_{\{U<\psi_{y}^-\}} \mathrm{div}( \nabla F_k(\nabla \psi_{y}^-) )\,(\psi_{y}^--U)\, dx\\
&+\frac{2}{k}\,\int_{\{U<\psi_{y}^-\}} \Delta \psi_{y}^- \, (\psi^-_{y}-U)\, dx.
\end{split}
\]
By noticing that \(\nabla\psi_{y}^- = \zeta_{y}^-+T\,\nabla a_{k,y}\) with \(|\zeta_{y}^-|\leq K+2\), the definition of \(T\) and the lower bound in \eqref{2fois} imply
\[
R+1< T\,|\nabla a_{k,y}|-(K+2)\leq \left|\nabla\psi_{y}^-\right|.
\] 
As a consequence of\footnote{If \(S_1\) and \(S_2\) are nonnegative symmetric matrices and \(\lambda\) is the lowest eigenvalue of \(S_1\), then \(\mathrm{tr\,}(S_1\,S_2) \geq \lambda\, \mathrm{tr\,}(S_2)\).} \eqref{eqFkuc1} and by construction of $\psi_y^-$, we have
\[
\begin{split}
\mathrm{div}(\nabla F_k(\nabla \psi_{y}^-))=\sum_{i,j=1}^N \partial_{x_i x_j}F_k(\nabla \psi_{y}^-(x))\, \partial_{x_i x_j} \psi_{y}^-&\geq \frac{\mu}{36}\, \Delta \psi_{y}^-\\
&=\frac{\mu}{36}\, T\,\Delta a_{k,y}=\frac{\mu}{18}\, T\ge \Lambda+1\ge f^+_k+1.
\end{split}
\]
Hence,
\[
\begin{split}
\int_{\{U<\psi_{y}^-\}} \mathrm{div}( \nabla F_k(\nabla \psi_{y}^-) )\,(\psi_{y}^--U )\,dx&+\frac{2}{k}\,\int_{\{U<\psi_{y}^-\}} \Delta \psi_{y}^- \, (\psi^-_{y}-U)\, dx\\
&\geq   \int_{\{U<\psi_{y}^-\}}(f^{+}_k+1)\,(\psi_{y}^{-}-U)\,dx.
\end{split}
\]
In view of \eqref{passagelimite}, we thus obtain
\[
\begin{split}
\int_{\{U<\psi_{y}^-\}} \big[F_k(\nabla U) -F_k(\nabla \psi_{y}^-)\big]\, dx &+\frac{1}{k}\, \int_{\{U<\psi_y^-\}} \Big[|\nabla U|^2-|\nabla \psi^-_y|^2\Big]\,dx \\
&\geq  \int_{\{U<\psi_{y}^-\}}(f^{+}_k+1) \,(\psi_{y}^--U )\,dx. 
\end{split}
\] 
But \eqref{eq180} then implies
\[
\left|\{U<\psi_{y}^-\}\right|=0,\qquad \mbox{ namely, for a.e. } x\in \mathcal{O},\quad U(x)\geq \psi_{y}^-(x).
\]
We now define the map $\ell^-$ as follows
\[
\ell^-(x)=\left\{\begin{array}{rl}
\sup\limits_{y\in \partial \mathcal{O}}\psi_{y}^-(x), & \mbox{ if } x\in\overline{\mathcal{O}},\\
&\\
\varphi_k(x),& \mbox{ otherwise}.
\end{array}
\right.
\]
Since \(\ell^-\) coincides with \(\varphi_k\) on \(\mathbb{R}^N \setminus \mathcal{O}\), is $(L_0/\mu)-$Lipschitz 
and satisfies \(\ell^-(x)\leq U(x)\) for every \(x\in \mathcal{O}\), this map has the desired properties.
\end{proof}

\subsection{Step 2: construction of competitors}
We still denote by \(U\)  a solution of \eqref{minimisationk} and we extend it by \(\varphi_k\) outside \(\mathcal{O}\). Let $L_0$ be the constant appearing in Proposition \ref{lm:barrier}. We pick \(\alpha\geq (L_0/\mu)\) and \(\tau\in \mathbb{R}^N\setminus\{0\}\) such that \(\mathcal{O} \cap (\mathcal{O}-\tau)\not=\emptyset\). For every function $\psi$, we denote 
\[
\psi_\tau(x)=\psi(x+\tau),\qquad x\in\mathbb{R}^N,
\]
and we set
\[
\psi_{\tau,\alpha}(x)=\psi_\tau(x) -\alpha\, |\tau|,\qquad x\in\mathbb{R}^N.
\]
Finally, we  introduce the notation \(\mathcal{O}_\tau:=\mathcal{O}-\tau\) and consider the functional
\[
\mathcal{F}_{k,\tau}(w)=\int_{\mathcal{O}_\tau} \Big[F_k(\nabla w(x)) + f_{k,\tau}(x)\,w(x)\Big] \,dx,\qquad w\in W^{1,1}(\mathcal{O}_\tau).
\]
By a change of variables, we have the following.
\begin{lm}\label{lemutsol}
The map \(U_{\tau,\alpha}\) minimizes \(\mathcal{F}_{k,\tau}\) on \(U_{\tau,\alpha}+W^{1,1}_0(\mathcal{O}_{\tau})\).
\end{lm}
\begin{comment}
\begin{proof} Let \(w\in U_{\tau,\alpha}+W^{1,1}_0(\mathcal{O}_{\tau})\). With the previous notation we consider $w_{-\tau,-\alpha}$, which belongs to \(U+W^{1,1}_0(\mathcal{O})\). Hence, by minimality of $u$ we have \(\mathcal{F}(u)\leq \mathcal{F}(w_{-\tau,-\alpha})\). By an obvious change of variables, this implies
\[
\begin{split}
\int_{\mathcal{O}_\tau} \Big[F_k(\nabla U_\tau(y)) +f_\tau(y)\, U_\tau(y)\Big] \,dy &\leq \int_{\mathcal{O}_\tau} \Big[F_k(\nabla w_{-\tau,-\alpha}(y+\tau)) +f_\tau(y)\, w_{-\tau,-\alpha}(y+\tau)\Big] \,dy\\
&= \int_{\mathcal{O}_\tau} \Big[F_k(\nabla w) +f_\tau\, w\Big] \,dy+\alpha\,|\tau|\,\int_{\mathcal{O}_\tau} f_\tau\, dy.
\end{split}
\]
On the other hand, by observing that 
\[
\int_{\mathcal{O}_\tau} \Big[F_k(\nabla U_\tau) +f_\tau\, U_\tau\Big]\,dx =\int_{\mathcal{O}_\tau} \Big[F_k(\nabla U_{\tau,\alpha}) +f_\tau\, U_{\tau,\alpha}\Big]\,dx +\alpha\, |\tau|\,\int_{\mathcal{O}_\tau} f_\tau\,dx,
\]
we get the desired conclusion.
\end{proof}
\end{comment}
In order to compare \(U\) and \(U_{\tau,\alpha}\) on \(\mathcal{O} \cap \mathcal{O}_\tau\), we will use the following result.
\begin{lm}\label{lemcompuut}
With the previous notation, we have
\[
U_{\tau,\alpha}(x) \leq U(x),\qquad \mbox{ for a.e. } x\in \mathbb{R}^N \setminus (\mathcal{O} \cap \mathcal{O}_\tau).
\]
\end{lm}
\begin{proof} Let \(x\in \mathbb{R}^N \setminus (\mathcal{O} \cap \mathcal{O}_\tau)\) be a Lebesgue point of \(U\) and \(U_\tau\).
\par
We first consider the case \(x\not\in \mathcal{O}\). By using \(U\leq \ell^+\) on \(\mathbb{R}^N\), \(\alpha\geq (L_0/\mu)\) and the $(L_0/\mu)-$Lipschitz continuity of $\ell^+$, we get
\[
U(x)=\varphi_k(x)=\ell^+(x)\geq \ell^+_\tau(x) -\frac{L_0}{\mu}\,|\tau|\geq U_\tau(x)-\frac{L_0}{\mu}\,|\tau|\geq U_{\tau,\alpha}(x). 
\]
If \(x\not\in \mathcal{O}_{\tau}\), we use \(U\geq \ell^-\) on $\mathbb{R}^N$ to get
\[
U(x)\geq \ell^-(x)\geq \ell^-_\tau(x)-\frac{L_0}{\mu}\,|\tau|=(\varphi_k)_\tau(x)-\frac{L_0}{\mu}\,|\tau|=U_\tau(x)-\frac{L_0}{\mu}\,|\tau|\geq U_{\tau,\alpha}(x).
\]
This completes the proof.
\end{proof}
We now introduce the two functions
\begin{equation}
\label{eqdefvw}
V=\left\{\begin{array}{rl} \max\{U,\,U_{\tau,\alpha}\}, &\mbox{ on } \mathcal{O}\cap \mathcal{O}_\tau, \\ U,&\mbox{ on } \mathcal{O}\setminus \mathcal{O}_\tau,\end{array}\right.\qquad 
W=\left\{\begin{array}{rl} \min\{U,\,U_{\tau,\alpha}\}& \mbox{ on } \mathcal{O}\cap \mathcal{O}_\tau, \\ U_{\tau,\alpha},&\mbox{ on } \mathcal{O}_\tau\setminus \mathcal{O}.\end{array}\right.
\end{equation}
Then by Lemma \ref{lemcompuut}, we have \(V\in U+W^{1,1}_0(\mathcal{O})\) and \(W\in U_{\tau,\alpha}+W^{1,1}_0(\mathcal{O}_\tau)\). By using the minimality of $U$, we have
\[
\mathcal{F}_k(U)\leq \mathcal{F}_k(\theta\, V+(1-\theta)\,U),\qquad \mbox{ for every }\theta\in[0,1].
\]
Analogously, by Lemma \ref{lemutsol} we get 
\[
\mathcal{F}_{k,\tau}(U_{\tau,\alpha})\leq \mathcal{F}_{k,\tau}(\theta\, W + (1-\theta)\,U_{\tau,\alpha}),\qquad \mbox{ for every }\theta\in[0,1].
\] 
By summing these two inequalities, and taking into account the definitions of \(V\) and \(W\), with elementary manipulations we finally obtain
\begin{equation}
\label{towards}
\begin{split}
\int_{A_{\tau}(\alpha)} \Big[F_k(\nabla U_{\tau,\alpha}) + F_k(\nabla U) &-F_k(\theta\, \nabla U_{\tau,\alpha} + (1-\theta)\, \nabla U)-F_k(\theta\, \nabla U + (1-\theta)\, \nabla U_{\tau,\alpha})\Big]\, dx \\ 
&\leq \theta\, \int_{A_{\tau}(\alpha)}(f_k-(f_k)_\tau)\,(U_{\tau,\alpha}-U)\,dx,
\end{split}
\end{equation}
where
\[
A_\tau(\alpha)=\left\{x\in \mathcal{O}\cap \mathcal{O}_{\tau}\, :\, \frac{U_{\tau,\alpha}(x)-U(x)}{|\tau|}\ge 0\right\}.
\]
\subsection{Step 3: a uniform Lipschitz estimate}

In this part we prove the following.
\begin{prop}\label{lemuniflipbd}
With the previous notation, we have
\[
\|\nabla U\|_{L^{\infty}(\mathcal{O})}\leq \frac{\mathcal{L}}{\mu},
\]
for some constant $\mathcal{L}=\mathcal{L}(N,R,K,\|f\|_{L^\infty(\Omega)},\mathrm{diam}(\Omega))>0$. In particular, the estimate is independent of $k\in\mathbb{N}$.
\end{prop}
\begin{proof}
Let us fix \( k\in \mathbb{N}\). We define the set
\[
A_{\tau}(\alpha,R)=\{x\in A_\tau(\alpha)\, :\, |\nabla U_{\tau,\alpha}(x)|\geq 2\,R\ \mbox{ and }\ |\nabla U(x)|\geq 2\,R \}.
\]
By using equation \eqref{eq1834} from Lemma \ref{lemma_uc} in inequality \eqref{towards}, and then dividing by \(\theta\) and letting \(\theta\) go \(0\), we get
\[
c\,\mu\,\int_{A_{\tau}(\alpha,R)}|\nabla U-\nabla U_{\tau,\alpha}|^2\, dx\leq \int_{A_\tau(\alpha)}(f_k-(f_k)_\tau)\,(U_{\tau,\alpha}-U)\,dx, 
\]
for some universal constant $c>0$.
We now assume that \(\tau=h\,e_1\) where \(h>0\) and \(e_1\) is the first vector of the canonical basis of \(\mathbb{R}^N\). Remember that \(f_k\) is compactly supported and define for almost every \(x=(x_1, \dots, x_N)\in \mathbb{R}^N\)
\[
g_{k,h}(x_1,\dots,x_N)=\int_{x_1}^{x_1+h} f_k(t, x_2, \dots, x_N)\,dt.
\]
Observe that \(g_{k,h}\) is a smooth compactly supported function. 
Moreover, we have that $g_{k,h}/h$ converges uniformly to $f_k$, since
\[
\left|\frac{g_{k,h}(x)}{h}-f_k(x)\right|\le \frac{1}{h}\,\int_{x_1}^{x_1+h} |f_k(t,x_2,\dots,x_N)-f(x_1,x_2,\dots,x_N)|\,dt,
\]
and $f_k$ is smooth and compactly supported.
By Fubini theorem, we also have for every \(\eta\in W^{1,1}(\mathbb{R}^N)\)
\[
\int_{\mathbb{R}^N} g_{k,h}\, \eta_{x_1}\,dx = \int_{\mathbb{R}^N} \big(f_k-(f_k)_{h\,e_1}\big)\,\eta\,dx.  
\]
We apply this observation to the map \(\eta=V-U\) (extended by \(0\) outside \(\mathcal{O}\)) 
where \(V\) is defined in \eqref{eqdefvw}. Since \(\eta\) coincides with $U_{h\,e_1, \alpha}-U$ on the set $A_{h\,e_1}(\alpha)$, we obtain
\[
\begin{split}
\int_{A_{h\,e_1}(\alpha)}(f_k-(f_k)_{h\,e_1})\,(U_{h\,e_1,\alpha}-U) \,dx
&= \int_{\mathbb{R}^N} g_{k,h}\, \left((U_{h\,e_1,\alpha}-U)_+\right)_{x_1}\,dx\\
&=\int_{A_{h\,e_1}(\alpha)} g_{k,h}\, \left(U_{h\,e_1,\alpha}-U\right)_{x_1}\,dx.
\end{split}
\]
Observe that $(U_{h\,e_1,\alpha})_{x_1}=(U_{h\,e_1})_{x_1}$.
If we commute the derivative and the translation, by dividing by a factor $h^2$ we then get
\begin{equation}\label{eq816}
\int_{A_{h\,e_1}(\alpha, R)} \left|\frac{\nabla U_{h\,e_1,\alpha}- \nabla U}{h}\right|^2\,dx\le \frac{C}{\mu}\, \int_{A_{h\,e_1(\alpha)}} \frac{g_{k,h}}{h}\, \frac{(U_{x_1})_{h\,e_1}-U_{x_1}}{h}\,dx.
\end{equation}
Now, since \(\alpha\geq L_0/\mu>2R\) (by definition of \(L_0\) in \eqref{L_0} and the fact that \(0<\mu\leq 1\)), 
\[ 
1_{\{U_{x_1}>\alpha\}} \leq \liminf_{h\to 0} 1_{A_{h\,e_1}(\alpha,R)}\leq \limsup_{h\to 0} 1_{A_{h\,e_1}(\alpha)} \leq 1_{\{U_{x_1}\geq \alpha\}}.
\]
Here, we also use the fact that \(U\) is smooth on \(\overline{\mathcal{O}}\).
This implies that we can apply the dominated convergence theorem and let \(h\) go to \(0\) in \eqref{eq816}, so to get
\[
\int_{\{U_{x_1}>\alpha\}} |\nabla U_{x_1}|^2\, dx^2\,dx\le \frac{C}{\mu}\, \int_{\{U_{x_1}\ge \alpha\}} f_k\, U_{x_1\,x_1}\,dx.
\]
From this and the fact that $\nabla U_{x_1}=0$ almost everywhere on $\{U_{x_1}=\alpha\}$, we get
\[
\int_{\{U_{x_1}>\alpha\}} |\nabla U_{x_1}|^2\, dx\leq \frac{C\Lambda}{\mu} \left(\int_{\{ U_{x_1} > \alpha\}} |U_{x_1x_1}|\,dx\right),
\]
which implies by H\"older inequality
\begin{equation}
\label{eq827}
\|\nabla U_{x_1}\|_{L^2(\{U_{x_1} >\alpha\})} \leq C\,\frac{\Lambda}{\mu}\, |\{U_{x_1} >\alpha\}|^{\frac{1}{2}}.
\end{equation}
We denote \(\Theta(\alpha)=|\{x\in\mathcal{O}\, :\, U_{x_1}>\alpha\}|\) the distribution function of $U_{x_1}$. 
By Cavalieri principle, we have
\[
\begin{split}
\|U_{x_1}-\alpha\|_{L^1(\{U_{x_1}>\alpha\})}&=\int_{\alpha}^{+\infty} |\{x\in\mathcal{O}\, :\, U_{x_1}>s\}|\, ds=:\int_\alpha^{+\infty} \Theta(s)\, ds.
\end{split}
\]
By H\"older and Sobolev inequalities and using \eqref{eq827}, we obtain the following inequality for almost every  \(\alpha\geq L_0/\mu\),
\[
\int_{\alpha}^{+\infty}\Theta(s)\,ds \leq \frac{C}{\mu}\,\Theta(\alpha)^{\gamma},
\]
with \(\gamma=(N+1)/N>1\) and \(C=C(N,\Lambda)>0\).
In other words, the nonnegative nonincreasing function \(\chi(\alpha)= \int_{\alpha}^{+\infty} \Theta(s)\,ds\) satisfies the following differential inequality 
\[
\chi(\alpha) \leq \frac{C}{\mu}\,\left(-\chi'(\alpha)\right)^\gamma,\qquad \mbox{ for a.\,e. } \alpha\geq \frac{L_0}{\mu},
\]
where \(C=C(N,\Lambda)>0\). This easily implies (as in Gronwall Lemma) that \(\chi(\alpha)=0\) for every \(\alpha\geq \alpha_0\) where
\[
\alpha_0=\frac{L_0}{\mu}+\frac{\gamma}{\gamma-1}\, \left(\frac{C}{\mu}\right)^{\frac{1}{\gamma}}\,\left(\int_{L_0/\mu}^{+\infty} \Theta(s)\,ds\right)^\frac{\gamma-1}{\gamma}.
\]
Since \(\chi\left(L_0/\mu\right)\leq (C/\mu)\,\Theta\left(L_0/\mu\right)^{\gamma}\leq (C/\mu)\,|\mathcal{O}|^{\gamma}\), we get
\[
\alpha_0 \leq \frac{L_0}{\mu}  + \frac{C}{\mu}\,\frac{\gamma}{\gamma-1}\,|\mathcal{O}|^{\gamma-1}.
\]
Observe that by definition we have $|\mathcal{O}|=|\Omega_k|$. Then by Lemma \ref{lemma_approximation_sets} and the isodiametric inequality, up to choosing $k$ sufficiently large we can suppose that $|\mathcal{O}|=|\Omega_k|\le C\,(\mathrm{diam}(\Omega)+1)^N$, for a constant $C=C(N)>0$. We thus have
\[
\alpha_0\le \frac{L_0}{\mu}  + \frac{C}{\mu}\,\frac{\gamma}{\gamma-1}\,\Big(\mathrm{diam}(\Omega)+1\Big)^{(\gamma-1)N}=:\frac{\mathcal{L}}{\mu},
\]
possibly for a different constant $C=C(N,\Lambda)>0$.
Observe that $\mathcal{L}=\mathcal{L}(N,K,R,\Lambda,\mathrm{diam}(\Omega))>0$. Thus we have \(\Theta(\mathcal{L}/\mu)=0\) and hence
\[
U_{x_1}\leq \frac{\mathcal{L}}{\mu}\, \qquad \mbox{ for a.\,e. }x\in \mathcal{O}.
\]
By reproducing the previous proof with $h<0$, we can show that \(U_{x_1}\geq -\mathcal{L}/\mu\) as well. In the end \(|U_{x_1}|\leq \mathcal{L}/\mu\) and of course this is true for every partial derivative, thus \(\|\nabla U\|_{L^\infty(\mathcal{O})}\leq \sqrt{N}\mathcal{L}/\mu\).
\end{proof}
\subsection{Step 4: passage to the limit}
Since we have to pass to the limit as $k$ goes to $\infty$, it is now convenient to  come back to the original notation $u_k$ and $\Omega_k$.
Let us set 
\[
\widetilde u_k:={u_k}_{|\overline{\Omega}},\qquad k\in\mathbb{N}.
\]
By Lemma \ref{lemuniflipbd} and Ascoli-Arzel\`a Theorem, there exists a subsequence of \(\{\widetilde u_k\}_{k\in \mathbb{N}}\) (we do not relabel) which uniformly converges to some map \(v\) on \(\overline{\Omega}\). Moreover, \(v\) is still $(\mathcal{L}/\mu)-$Lipschitz continuous on $\overline\Omega$. 
\begin{lm}
\label{lemconv}
The limit function \(v\) solves \eqref{minimisation_weak}.
\end{lm}
\begin{proof}
We first prove that \(v\) agrees with \( \varphi\) on \(\partial \Omega\). Let \(y\in \partial \Omega\), by Lemma \ref{lemma_approximation_sets} there exists a sequence of points \(y_k \in \partial \Omega_k\) converging to $y$.
Then
\[
|v(y)-\varphi(y)|\leq |v(y)-u_k(y)|+|u_k(y)-u_k(y_k)| + |\varphi_k(y_k)-\varphi_k(y)| .
\]
Here, we have used the fact that \(u_k|_{\partial\Omega_k}= \varphi_k|_{\partial\Omega_k}\) and \(\varphi_k|_{\partial \Omega}=\varphi\). Since the Lipschitz constants of \(u_k\) and \(\varphi_k\) can be bounded independently of \(k\), this implies \(v(y)=\varphi(y)\). Hence, \(v=\varphi\) on \(\partial \Omega\).
\vskip.2cm\noindent
We now prove that \(v\) minimizes $\mathcal{F}$ in \( \varphi+W^{1,2}_0(\Omega)\).
 Let \(w\in \varphi+W^{1,2}_0(\Omega)\), we denote by \(w_k\) the extension of \(w\) by \(\varphi_k\) out of \(\Omega\). Observe that we have \(w_k\in \varphi_k+  W^{1,2}_{0}(\Omega_k)\).
Then for every \(k\in \mathbb{N}\) we have \(\mathcal{F}_k(u_k)\leq \mathcal{F}_k(w_k)\), which gives
\begin{equation}
\label{comparison}
\begin{split}
\int_{\Omega_k} F_k(\nabla u_k) \,dx+ \int_{\Omega_k} f_k\,u_k\,dx\le \int_{\Omega_k} F_k(\nabla w_k) \,dx+\frac{1}{k}\,\int_{\Omega_k} |\nabla w_k|^2\,dx + \int_{\Omega_k} f_k\,w_k\,dx.
\end{split}
\end{equation}
By definition of $w_k$ and the fact that $F_k\leq F$, we get
\[
\begin{split}
\int_{\Omega_k} F_k(\nabla w_k)\,dx \leq \int_{\Omega} F(\nabla w)\,dx + \int_{\Omega_k\setminus\Omega}F(\nabla \varphi_k)\,dx.
\end{split}
\]
Since \(\varphi_k\) is $(K+1)-$Lipschitz , this gives
\[
\int_{\Omega_k}F_k(\nabla w_k) \,dx\leq \int_{\Omega}F(\nabla w) \,dx+ \,|\Omega_k\setminus \Omega| \,\max_{\xi \in B_{K+1}} F(\xi).
\]
By using Lemma \ref{lemma_approximation_sets}, we thus obtain
\[
\limsup_{k\to +\infty}\int_{\Omega_k}F_k(\nabla w_k)\,dx\le \int_{\Omega}F(\nabla w)\,dx .
\]
We also observe that
\begin{equation}
\label{dessous}
\begin{split}
F_k(\nabla u_k(x))&\geq   F(\nabla u_k(x)) - |F_k(\nabla u_k(x)) - F(\nabla u_k(x))| \\
&\geq F(\nabla u_k(x))-\max_{\{\xi\,:\, |\xi|\le \mathcal{L}/\mu\}} |F_k(\xi) - F(\xi)|.
\end{split}  
\end{equation}
where we used Proposition \ref{lemuniflipbd}.
The sequence \(\{F_k\}_{k\in \mathbb{N}}\) uniformly converges to \(F\) on bounded sets, thus using  \eqref{dessous} we get
\begin{equation}\label{eq792}
\begin{split}
\liminf_{k\to +\infty}\int_{\Omega_k}F_k(\nabla u_k)\,dx&\geq \liminf_{k\to +\infty}\left(\int_{\Omega}F(\nabla u_k)\,dx -|\Omega_k|\,\|F_k-F\|_{L^{\infty}(B_{\mathcal{L}/\mu})} \right)\\ &\geq \liminf_{k\to +\infty}\int_{\Omega}F(\nabla u_k)\,dx \geq \int_{\Omega}F(\nabla v)\,dx.
\end{split}
\end{equation}
In the last inequality we used the weak lower semicontinuity of the functional \(w\mapsto \int_{\Omega} F(\nabla w)\) on \(W^{1,1}(\Omega)\).
 Clearly, 
\begin{equation}\label{eq797}
\lim_{k\to \infty} \frac{1}{k} \,\int_{\Omega_k}|\nabla w_k|^2\, dx=\lim_{k\to +\infty}\left[\frac{1}{k}\int_{\Omega}|\nabla w|^2\,dx +\frac{1}{k}\int_{\Omega_k\setminus \Omega} |\nabla \varphi_k|^2\,dx\right]=0.
\end{equation}
By recalling \eqref{norme} and that $|\Omega_k\setminus\Omega|$ converges to $0$, we have
\[
\lim_{k\to\infty}\left|\int_{\Omega_k\setminus \Omega} f_k\,w_k\,dx\right|\le \Lambda\,\lim_{k\to\infty}\|\varphi_k\|_{L^{1}(\Omega_k\setminus\Omega)}=0.
\] 
By using this fact, the $\ast-$weak convergence of $f_k$ to $f$ in $L^\infty(\Omega)$ and that $w_k=w$ on $\Omega$, we get
\begin{equation}\label{eq808}
\begin{split}
\lim_{k\to \infty} \int_{\Omega_k}f_k\, w_k\,dx=\lim_{k\to\infty} \int_{\Omega} f_k\, w\,dx=\int_\Omega f\,w\,dx.
\end{split}
\end{equation}
Moreover, since \(\{\widetilde u_k\}_{k\in \mathbb{N}}\) converges to \(v\) in \(L^1(\Omega)\), with a similar argument we also have
\begin{equation}
\label{eq812}
\lim_{k\to \infty} \int_{\Omega_k}f_k\,u_k\,dx = \int_\Omega f\,v\,dx.
\end{equation}
By passing to the limit in \eqref{comparison} and using \eqref{eq792}, \eqref{eq797}, \eqref{eq808} and \eqref{eq812}, we get $\mathcal{F}(v)\le \mathcal{F}(w)$. By arbitrariness of $w$, this shows that $v$ is a solution of \eqref{minimisation_weak}.
\end{proof}

\subsection{Step 5: conclusion}
Since \(v\) is a $(\mathcal{L}/\mu)-$Lipschitz solution of \eqref{minimisation_weak}, we use Lemma \ref{propagationofregularity} to conclude that every solution $u$ of \eqref{minimisation_weak} is Lipschitz continuous. More precisely, we have the following estimate
\[
\|\nabla u\|_{L^\infty(\Omega)}\le 2\,R+\|\nabla v\|_{L^\infty(\Omega)}\le \frac{(2\,R+\mathcal{L})}{\mu}=:\frac{\mathcal{L}_0}{\mu},
\]
where we used again that $0<\mu\le 1$.
This completes the proof of Theorem \ref{teo:minimisationk}.

\section{Proof of the Main Theorem}
\label{sec:5}

Finally, we come to the proof of the Main Theorem. 
We will need the following ``density in energy'' result, whose proof can be found in \cite[Theorem 4.1]{BMT}. The original statement is indeed fairly more general, we give a version adapted to our needs.
\begin{lm}[\cite{BMT}]
\label{lm:BMT}
Let $\Omega\subset\mathbb{R}^N$ be an open bounded convex set. Let $F:\mathbb{R}^N\to\mathbb{R}$ be a convex function, $f\in L^N(\Omega)$ and $\varphi:\mathbb{R}^N\to\mathbb{R}$ a Lipschitz continuous function. Then for every $u\in\varphi+W^{1,1}_0(\Omega)$ such that 
\[
\int_\Omega |F(\nabla u)|\,dx<+\infty,
\]
there exists a sequence $\{u_k\}_{k\in\mathbb{N}}\subset \varphi+W^{1,\infty}_0(\Omega)$ such that
\begin{equation}
\label{pierre}
\lim_{k\to\infty} \|u_k-u\|_{W^{1,1}(\Omega)}=0\qquad \mbox{ and }\qquad \lim_{k\to\infty} \int_\Omega \Big[F(\nabla u_k)+f\,u_k\Big]\,dx=\int_\Omega\Big[ F(\nabla u)\,+f\,u\Big]\,dx.
\end{equation}
\end{lm}
\begin{proof}
For the case $f\equiv 0$, the proof is contained in \cite{BMT}. In order to cover the case $f\in L^{N}(\Omega)$, it is sufficient to observe that by Sobolev embedding and strong convergence in $W^{1,1}(\Omega)$, the sequence $\{u_{k}\}_{k\in\mathbb{N}}$ also converges weakly in $L^{N'}(\Omega)$. Then the result easily follows.
\end{proof}

\subsection{Step 1: reduction to $\mu-$uniformly convex problems}

For every $Q>R$, we consider the minimization problem
\begin{displaymath}
\tag{$P_Q$}
\label{minimisationQ}
\min\left\{\int_{\Omega} \Big[F_Q(\nabla u) +f\, u\Big] \,dx\, :\, u-\varphi\in W^{1,1}_0(\Omega)\right\},
\end{displaymath}
where $F_Q:\mathbb{R}^N\to\mathbb{R}$ is a convex function such that 
\begin{enumerate}
\item[{\it i)}] $F_Q\equiv F$ in $B_Q$;
\vskip.2cm
\item[{\it ii)}] $F$ is $\mu_Q-$uniformly convex outside the ball $B_R$, with 
\[
\mu_Q=\min\left\{1,\min_{t\in[2\,R,4\,Q]} \Phi(t)\right\}.
\]
\end{enumerate}
see Lemma \ref{lm:pierresque} below. Observe that $0<\mu_Q\le 1$. By Theorem \ref{teo:minimisationk} and \eqref{lipschitz}, we get that every minimizer $u_Q$ of \eqref{minimisationQ} is such that
\[
\|\nabla u_Q\|_{L^\infty(\Omega)}\le \frac{\mathcal{L}_0}{\mu_Q},
\] 
with $\mathcal{L}_0$ independent of $Q$. We now take $Q\gg R$ sufficiently large so that
\[
\frac{\mathcal{L}_0}{\mu_Q}\le Q-1.
\]
Observe that this is possible, thanks to the definition of $\mu_Q$ and property \eqref{superlinear} of the function $\Phi$. Thanks to this choice and the construction of the function $F_Q$, we thus get
\begin{equation}
\label{etun}
F_Q(\nabla u_Q) =F(\nabla u).
\end{equation}
Let us take $v\in \varphi+W^{1,\infty}_0(\Omega)$
and $\theta\in(0,1)$ such that
\[
\theta\, \left(Q-1-\|\nabla v\|_{L^\infty(\Omega)}\right)+\|\nabla v\|_{L^{\infty}(\Omega)}\le Q.
\]
Then the function $\theta\,u_Q+(1-\theta)\,v$ is such that
\[
\|\theta\,\nabla u_Q+(1-\theta)\,\nabla v\|_{L^\infty(\Omega)}\le Q,
\]
so that by using again $F\equiv F_Q$ in the ball $B_Q$ and the convexity of $F$, we get
\begin{equation}
\label{etdeux}
\begin{split}
F_Q(\theta\,\nabla u_Q+(1-\theta)\,\nabla v)&=F(\theta\,\nabla u_Q+(1-\theta)\,\nabla v)\\
&\le\theta\, F(\nabla u_Q)+(1-\theta)\, F(\nabla v).
\end{split}
\end{equation}
We observe that $\theta\,u_Q+(1-\theta)\,v$ is admissible for \eqref{minimisationQ}, then by using the minimality of $u_Q$, \eqref{etun} and \eqref{etdeux}, we obtain
\[
\int_{\Omega} F(\nabla u_Q) +f\, u_Q\,dx\le \theta\, \int_{\Omega} \Big[F(\nabla u_Q) +f\, u_Q\Big] \,dx+(1-\theta)\, \int_\Omega \Big[F(\nabla v) +f\, v \Big]\,dx,
\]
which finally shows that $u_Q$ minimizes the original Lagrangian $\mathcal{F}$ among Lipschitz functions, i.e. $u_Q$ is a solution of
\[
\min\left\{\int_{\Omega} \Big[F(\nabla u) +f\, u\Big] \,dx\, :\, u-\varphi\in W^{1,\infty}_0(\Omega)\right\}.
\]
\subsection{Step 2: conclusion by density}
In order to complete the proof, it is only left to prove that $u_Q$ minimizes $\mathcal{F}$ among $W^{1,1}$ functions as well.
At this aim, let $v\in \varphi+W^{1,1}_0(\Omega)$ be such that
\[
\int_{\Omega} \Big[F(\nabla v)+f\,v\Big]\,dx<+\infty.
\]
Since $f\,v\in L^1(\Omega)$, this implies that $\int_\Omega F(\nabla v)<+\infty$.
By noting $F_+$ and $F_-$ the positive and negative parts of $F$, we observe that by \eqref{M}, we can infer
\[
\int_\Omega F_-(\nabla v)\,dx\le \max_{B_r} |F|\, |\Omega|. 
\] 
This and the fact that
\[
+\infty>\int_\Omega F(\nabla v)\,dx=\int_{\Omega} F_+(\nabla v)\,dx-\int_\Omega F_-(\nabla v)\,dx,
\]
imply $F(\nabla v)\in L^1(\Omega)$.
By Lemma \ref{lm:BMT}, then there exists a sequence $\{v_k\}_{k\in\mathbb{N}}\subset\varphi+W^{1,\infty}_0(\Omega)$ such that \eqref{pierre} holds. 
By using the minimality of $u_Q$ in $\varphi+W^{1,\infty}_0(\Omega)$, we get
\[
\int_{\Omega} \Big[F(\nabla u_Q) +f\, u_Q\Big] \,dx\le \int_{\Omega} \Big[F(\nabla v_k) +f\, v_k\Big] \,dx.
\]
If we now pass to the limit and use \eqref{pierre}, we obtain that $u_Q$ is a Lipschitz solution of \eqref{minimisation}. By appealing again to Lemma \ref{propagationofregularity}, we finally obtain that every solution of \eqref{minimisation} is globally Lipschitz continuous. This concludes the proof of the Main Theorem.

\section{More general lower order terms}
\label{sec:6}

In this section, we consider more general functionals of the form 
\[
\mathcal{F}(u)=\int_{\Omega} \Big[F(\nabla u) + G(x,u)\Big] \,dx,
\]
where \(G:\Omega \times \mathbb{R} \to \mathbb{R}\) is a measurable function which satisfies the following assumption:
\begin{itemize}
\item[({\it HG})] 
\begin{itemize}
\item[$\bullet$] there exists  a positive \(g\in L^\infty(\Omega)\) such that 
\[
|G(x,u)-G(x,v)|\leq g(x)\,|u-v|,\qquad \mbox{ for a.\,e. }x\in\Omega, \mbox{ every }u,v\in\mathbb{R};
\]
\item[$\bullet$] there exists \(b\in L^{N'}(\Omega)\) such that 
\[
\int_\Omega |G(x, b(x))|\,dx<+\infty. 
\]
\end{itemize}
\end{itemize}
\begin{oss}
As a consequence of {\it (HG)}, we have that for every \(v\in L^{N'}(\Omega)\) (in particular for every \(v\in \varphi + W^{1,1}_0(\Omega)\)), the map \(x\mapsto G(x, v(x))\) is in \(L^{1}(\Omega)\).
\end{oss}
We then have the following generalization of the Main Theorem.
\begin{teo}
\label{coroglobreg}
Let \(\Omega\subset\mathbb{R}^N\) be a bounded convex open set, \(\varphi:\mathbb{R}^N \to \mathbb{R}\) a Lipschitz function, \(F:\mathbb{R}^{N}\to \mathbb{R}\) a convex function and $G:\Omega\to \mathbb{R}$ a map satisfying (HG). We consider the following problem
\begin{displaymath}
\tag{$P_G$}
\label{minimisationG}
\min\left\{\int_{\Omega} \Big[F(\nabla u) +G(x,u)\Big] \,dx\, :\, u-\varphi\in W^{1,1}_0(\Omega)\right\}.
\end{displaymath}
We assume that \(\varphi|_{\partial \Omega}\) satisfies the bounded slope condition of rank $K>0$ and that \(F\) is $\Phi-$uniformly convex outside some ball \(B_R\).
Then problem \eqref{minimisationG} admits a solution and every such  solution is Lipschitz continuous.
\end{teo}
\begin{proof}
Assumption \((HG)\) implies that for almost every \(x\in \Omega\) and every \(v\in \mathbb{R}\),
\[
G(x,v)\geq G(x,b(x))-g(x)\,|v-b(x)|\geq \Big[G(x,b(x))-g(x)\,|b(x)|\Big]\ -g(x)\,|v|
\]
and the function into square brackets is in \(L^{1}(\Omega)\). Thus we can apply Proposition \ref{prop:existence} and deduce existence of a solution  \(u\in W^{1,1}_0(\Omega)+\varphi\). 
We now divide the proof in two parts. 
\vskip.2cm\noindent
{\it Part I.} Let us first assume that for almost every \(x\in \Omega\), \(v\mapsto G(x, v)\) is differentiable. 
We denote by \(f\) the measurable map 
\[
f(x)=G_u(x, u(x)). 
\]
Observe that \(|f|\leq g\) almost everywhere on \(\Omega\).
For every \(v\in \varphi+W^{1,1}_0(\Omega)\) and every \(\theta\in (0,1)\), the minimality of $u$ implies 
\[
\mathcal{F}(u)\leq \mathcal{F}((1-\theta)\, u + \theta\, v).
\] 
Hence, by convexity of \(F\),
\[
\int_{\Omega} F(\nabla u) \,dx \leq  \int_{\Omega} F(\nabla v)\,dx + \int_{\Omega} 
\left[\frac{G(x, u+\theta\, (v-u))-G(x,u)}{\theta}\right]\,dx. 
\]
Thanks to {\it (HG)}, we can apply the dominated convergence theorem in the right hand side to get
\[
\int_{\Omega} \Big[F(\nabla u) +f\, u\Big]\,dx \leq  \int_{\Omega} \Big[F(\nabla v) + f\,v\Big]\,dx. 
\]
This proves that  \(u\) is a minimum of the initial problem \eqref{minimisation} to which the Main Theorem applies. In particular, \(u\) is $\mathcal{L}-$Lipschitz continuous, where \(\mathcal{L}\) depends on \(N, \Phi, K,R,\mathrm{diam}(\Omega)\) and \(\|g\|_{L^{\infty}(\Omega)}\).
This proves the statement under the additional assumption that  \(G\) is differentiable with respect to \(u\).
\vskip.2cm\noindent
{\it Part II.} In the general case, we introduce the sequence 
\[
G_{\varepsilon}(x,u)=\int_{\mathbb{R}}G(x,u-v)\,\rho_{\varepsilon}(v)\, dv,
\] 
where \(\rho_{\varepsilon}\) is a smooth convolution kernel.  Observe that \(G_{\varepsilon}\) satisfies {\it (HG)} with the same functions \(g\) and \(b\) and is differentiable with respect to \(u\). 
%Assumption {\it (HG)} again implies 
%that for almost every \(x\in \Omega\) and every \(v\in \mathbb{R}\),
%\[
%\begin{split}
%G_\varepsilon(x,v)&\geq \Big[G_\varepsilon(x,u(x))-g(x)\,|u(x)|\Big] -g(x)\,|v|.\\
%\end{split}
%\]
Hence, we can apply Proposition \ref{prop:existence} again to obtain a solution \(u_{\varepsilon}\) to 
\[
\min\left\{\int_{\Omega} \Big[F(\nabla u) +G_\varepsilon(x,u)\Big] \,dx\, :\, u-\varphi\in W^{1,1}_0(\Omega)\right\}.
\]
By Part I of the proof, we know that \(u_{\varepsilon}\) is $\alpha_1-$Lipschitz continuous, with $\alpha_1$ independent of \(\varepsilon\). Up to a subsequence, this net of minimizers thus converges to a Lipschitz continuous function, which solves \eqref{minimisationG}. This last assertion can be established as in Lemma \ref{lemconv}. In view of Lemma \ref{propagationofregularity}, this completes the proof of Theorem \ref{coroglobreg}.
\end{proof}

When \(\Omega\) is a uniformly convex set, every \(C^{1,1}\) map \(\varphi : \partial\Omega \to \mathbb{R}\) satisfies the bounded slope condition. In contrast, this condition becomes more restrictive when \(\partial \Omega\) contains affine faces. For instance, if \(\Omega\) is a convex polytope, the bounded slope condition requires  \(\varphi\) to be affine on each face of \(\partial \Omega\). In \cite{Cl}, Clarke has introduced the lower bounded slope condition, which can be seen as one half of the full two sided bounded slope condition. 

\begin{definition}
\label{definition_lbsc}
Let \(\Omega\) be a bounded open set in \(\mathbb{R}^N\) and $K>0$.
We say that a map \(\varphi : \partial \Omega \to \mathbb{R}\) satisfies the {\it lower bounded slope condition of rank $K$} if for every \(y\in \partial \Omega\), there exists \(\zeta_{y}^- \in \mathbb{R}^N\) such that \(|\zeta_{y}^-|\leq K\) and 
\begin{equation}
\label{eqLBSC}
\varphi(y)+\langle \zeta_{y}^-, x-y\rangle \leq \varphi(x) ,\qquad \mbox{ for every } x\in \partial \Omega.
\end{equation}
\end{definition}

It follows from the proof of Lemma \ref{lemmabsc} that a function \(\varphi : \partial \Omega \to \mathbb{R}\) satisfies the lower bounded slope condition if and only if it is the restriction to \(\partial \Omega\) of a convex function defined on  \(\mathbb{R}^N\).
The Main Theorem has the following variant when the bounded slope condition is replaced by the weaker lower bounded slope condition:

\begin{teo}
\label{teoreglbsc}
Let \(\Omega\) be a bounded convex  open set in \(\mathbb{R}^N\), \(\varphi:\mathbb{R}^N \to \mathbb{R}\) a Lipschitz continuous function, \(F:\mathbb{R}^{N}\to \mathbb{R}\) a convex function and $f\in L^\infty(\Omega)$. 
We assume that \(\varphi|_{\partial \Omega}\) satisfies the lower bounded slope condition of rank $K\ge 0$ and that \(F\) is $\mu-$uniformly convex outside some ball \(B_R\).
Then every solution \(u\) of \eqref{minimisation} is locally Lipschitz continuous on \(\Omega\).
\end{teo}
The proof follows the lines of the proof of the Main Theorem except that the translation technique in Step 2 of the proof of Theorem \ref{teo:minimisationk} must be replaced by the dilation technique introduced in \cite{Cl}. We omit the details.

\appendix

\section{Uniformly convex functions outside a ball}
\label{section_approximation_convex_functions}

\subsection{Basic properties}
\label{ucfoab}

We first present a characterization of uniformly convex functions outside a ball in terms  of second order derivatives.

\begin{lm}\label{lmunifconvsndder}
Let \(F:\mathbb{R}^N\to \mathbb{R}\) be a convex function and \(\{\rho_{\varepsilon}\}_{\varepsilon>0}\subset C^{\infty}_0(B_{\varepsilon})\) be a sequence of standard mollifiers. 
\begin{itemize}
\item[$(i)$] Assume that \(F\) is  \(\Phi-\)uniformly convex outside some ball \(B_R\). Then for every \(\varepsilon>0\), for every \(R'>R+\varepsilon\), for every \(\xi\in B_{R'}\setminus B_{R+\varepsilon}\) and \(\eta\in \mathbb{R}^N\),
\begin{equation}\label{eq461}
\langle \nabla^2(F*\rho_{\varepsilon})(\xi)\,\eta, \eta\rangle\geq \left(\min_{t\in [2\,R,2\,(R'+\varepsilon)]}\Phi(t)\right)\, |\eta|^2.
\end{equation}
\item[$(ii)$] Assume that there exist \(\mu>0\) and \(R>0\) such that for every \(\varepsilon>0\), for every \(|x|\geq R+\varepsilon\), for every \(\eta\in \mathbb{R}^N\),
\begin{equation}\label{eq465}
\langle \nabla^2(F*\rho_{\varepsilon})(x)\,\eta, \eta\rangle\geq \mu\, |\eta|^2.
\end{equation}
Then \(F\) is \(\mu-\)uniformly convex outside \(B_R\).
\end{itemize}
\end{lm}
\begin{proof}
Assume first that \(F\) is \(\Phi-\)uniformly convex outside \(B_R\). For every \(\xi\in B_{R'}\setminus B_{R+\varepsilon}\), for every \(y\in B_{\varepsilon}, \eta\in \mathbb{R}^N\) and every \(h>0\) sufficiently small, the segment \([\xi+h\,\eta-y, \xi-h\,\eta-y]\) does not intersect \(B_R\).
Hence
\[
\begin{split}
F(\xi-y)&\leq \frac{1}{2}F(\xi+h\,\eta-y) + \frac{1}{2}F(\xi-h\,\eta-y)-\frac{1}{2}\,h^2\,|\eta|^2\,\Phi(|\xi+h\,\eta-y|+|\xi-h\,\eta-y|)\\
&\leq \frac{1}{2}F(\xi+h\,\eta-y) + \frac{1}{2}F(\xi-h\,\eta-y)-\frac{1}{2}\, h^2\,|\eta|^2 \left(\min_{t\in [2\,R, 2\,(R'+\varepsilon+ h\,|\eta|)]}\Phi(t)\right).
\end{split}
\]
By multiplying by $\rho_\varepsilon(y)$ and integrating, this gives
\[
F*\rho_{\varepsilon}(\xi)\leq \frac{1}{2}F*\rho_{\varepsilon}(\xi+h\,\eta) + \frac{1}{2}F*\rho_{\varepsilon}(\xi-h\,\eta) -\frac{1}{2}\,h^2\,|\eta|^2\,\left(\min_{t\in [2\,R, 2\,(R'+\varepsilon+ h\,|\eta|)]}\Phi(t)\right).
\]
Hence, we get
\[
\begin{split}
\langle \nabla^2(F*\rho_{\varepsilon})(\xi)\,\eta, \eta\rangle &= \lim_{h\to 0}\frac{F*\rho_{\varepsilon}(\xi+h\,\eta) + F*\rho_{\varepsilon}(\xi-h\,\eta)-2\,F*\rho_{\varepsilon}(\xi)}{h^2}\\
&\geq \left(\min_{t\in [2\,R, 2\,(R'+\varepsilon)]}\Phi(t)\right)\,|\eta|^2.
\end{split}
\]
This completes the first part of the statement. 
\vskip.2cm\noindent 
Assume now that \eqref{eq465} holds true. Let \(\theta\in [0,1]\)  and \(\xi, \xi'\in \mathbb{R}^N\) be such that \([\xi,\xi']\cap B_{R} = \emptyset\). For every $\varepsilon>0$, we take $\xi_\varepsilon,\xi_\varepsilon'\in\mathbb{R}^N$ such that \([\xi_\varepsilon,\xi'_\varepsilon]\cap B_{R+\varepsilon} = \emptyset\) and 
\[
\lim_{\varepsilon\to 0} |\xi_\varepsilon-\xi|=0\qquad\mbox{ and }\qquad \lim_{\varepsilon\to 0} |\xi'_\varepsilon-\xi'|=0.
\]
We have
\[
\begin{split}
F*\rho_{\varepsilon}(\xi_\varepsilon)&=F*\rho_{\varepsilon}(\theta\,\xi_\varepsilon+(1-\theta)\,\xi'_\varepsilon)+(1-\theta)\,\langle \nabla (F*\rho_{\varepsilon})(\theta\,\xi_\varepsilon+(1-\theta)\,\xi'_\varepsilon), \xi_\varepsilon-\xi'_\varepsilon\rangle\\
& \quad +(1-\theta)^2\,\int_{0}^1(1-t)\,\langle D^2 (F*\rho_{\varepsilon})\big(t\,\xi_\varepsilon+(1-t)\,(\theta\,\xi_\varepsilon +(1-\theta)\,\xi'_\varepsilon)\big)(\xi_\varepsilon-\xi'_\varepsilon),\xi_\varepsilon-\xi'_\varepsilon\rangle\,dt.
\end{split}
\]
Since the segment \([\xi_\varepsilon, \theta\,\xi_\varepsilon+(1-\theta)\,\xi'_\varepsilon]\) does not intersect \(B_{R+\varepsilon}\), assumption \eqref{eq465} implies
\[
\begin{split}
F*\rho_{\varepsilon}(\xi_\varepsilon)\geq F*\rho_{\varepsilon}(\theta\,\xi_\varepsilon+(1-\theta)\,\xi'_\varepsilon)&+(1-\theta)\,\langle \nabla (F*\rho_{\varepsilon})(\theta\,\xi_\varepsilon+(1-\theta)\,\xi'_\varepsilon), \xi_\varepsilon-\xi'_\varepsilon\rangle\\
& + \frac{\mu}{2}\, (1-\theta)^2\,|\xi_\varepsilon-\xi'_\varepsilon|^2.
\end{split}
\]
Similarly, we get
\[
\begin{split}
F*\rho_{\varepsilon}(\xi_\varepsilon')\geq F*\rho_{\varepsilon}(\theta\,\xi_\varepsilon+(1-\theta)\,\xi'_\varepsilon)&+\theta\,\langle \nabla (F*\rho_{\varepsilon})(\theta\,\xi_\varepsilon+(1-\theta)\,\xi'_\varepsilon), \xi'_\varepsilon-\xi_\varepsilon\rangle\\
& + \frac{\mu}{2}\, \theta^2\,|\xi_\varepsilon-\xi'_\varepsilon|^2.
\end{split}
\]
We multiply the first inequality by \(\theta\) and the second one by \((1-\theta)\). By summing them, we get
\[
\theta\, F*\rho_{\varepsilon}(\xi_\varepsilon) + (1-\theta)\,F*\rho_{\varepsilon}(\xi'_\varepsilon) \geq  F*\rho_{\varepsilon}(\theta\,\xi_\varepsilon+(1-\theta)\,\xi'_\varepsilon) +\frac{\mu}{2}\,\theta\,(1-\theta)\,|\xi_\varepsilon- \xi'_\varepsilon|^2.
\]
We then let \(\varepsilon\) go to \(0\) to obtain \eqref{equc}. This completes the proof.
\end{proof}

We will also need the following technical result. 
Here $\mathcal{H}^1$ denotes the $1-$dimensional Hausdorff measure.
\begin{lm}
\label{lemunifconvh1}
Let \(F\) be a convex function which is   $\mu-$uniformly convex outside a ball \(B_R=B_R(0)\subset \mathbb{R}^N\). For every \(\xi, \xi' \in \mathbb{R}^N\) and every \(\zeta\in \partial F(\xi)\), we have
\begin{equation}
\label{eq1739}
F(\xi')\geq F(\xi) + \langle \zeta, \xi'-\xi \rangle + \frac{\mu}{4}\, \mathcal{H}^1\left([\xi, \xi']\setminus B_R\right)^2.
\end{equation}
\end{lm}
\begin{proof} We can assume that \(\xi\not=\xi'\).
We have several possible cases:
\vskip.2cm
\begin{itemize}
\item[{\it Case A.}] $\mathcal{H}^1([\xi,\xi']\setminus B_R)=\mathcal{H}^1([\xi,\xi'])$. In this case, the segment \([\xi,\xi']\) does not intersect \(B_R\) and thus \eqref{eq1838} follows directly from \eqref{equc_bis}.
\vskip.2cm
\item[{\it Case B.}] $\mathcal{H}^1([\xi,\xi']\setminus B_R)=0$. Then \eqref{eq1838}  follows from the convexity of $F$.
\vskip.2cm
\item[{\it Case C.}] $0<\mathcal{H}^1([\xi,\xi']\setminus B_R)<\mathcal{H}^1([\xi,\xi'])$. Without loss of generality, we may assume that \(\xi\not \in \overline{B_R}\) and that the half-line \(\{\xi+t(\xi'-\xi) ; t\geq 0\}\) intersects the sphere \(\partial B_R\) at two points \(\xi_1, \xi_2\) such that \(\xi_1\in [\xi_2, \xi]\).
We now in turn have to consider two cases: 
\[
\xi' \in [\xi_1, \xi_2]\qquad \mbox{ or }\qquad \xi_2\in [\xi', \xi_1],
\]
(see Figure \ref{fig:1} below).
\begin{figure}[h]
\includegraphics[scale=.3]{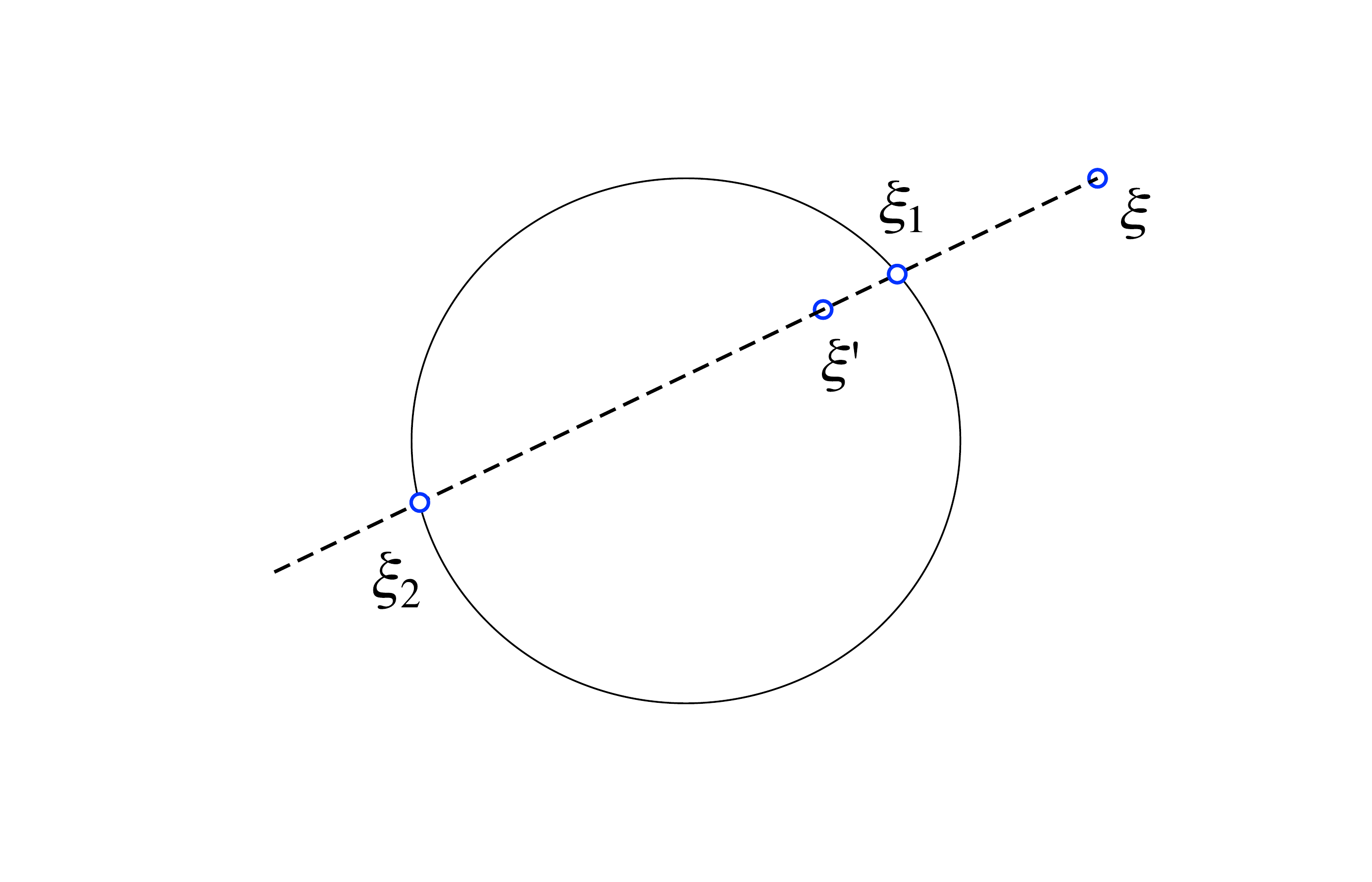}\includegraphics[scale=.3]{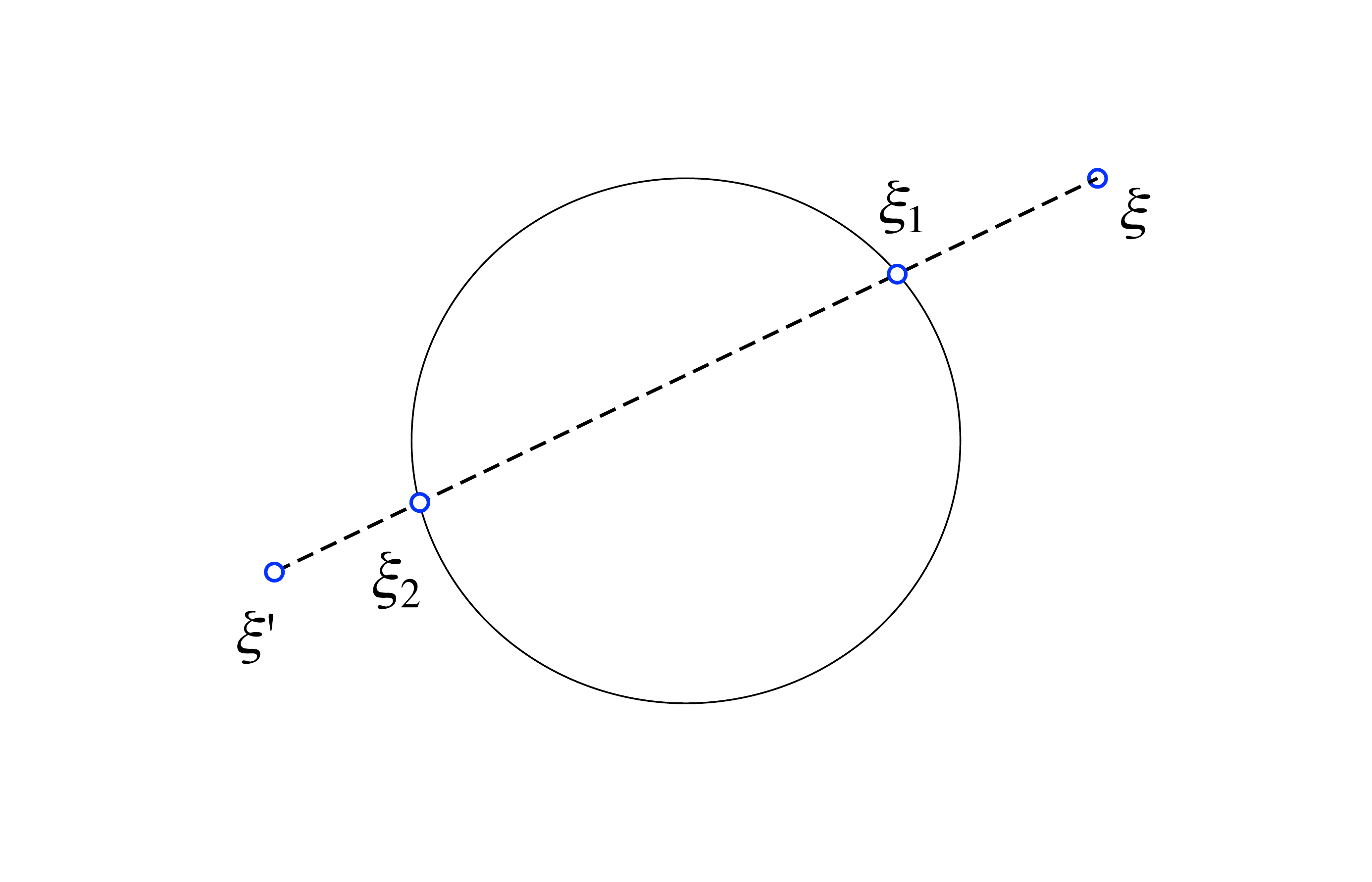}
\caption{The two possibilities for {\it Case C} in the proof of Lemma \ref{lemunifconvh1}.}
\label{fig:1}
\end{figure}
\vskip.2cm
When \(\xi' \in [\xi_1, \xi_2]\), we use the fact that the segment \([\xi_1, \xi]\) lies outside the ball \(B_R\). It follows from \eqref{equc_bis} that for every \(\zeta_1\in \partial F(\xi_1)\),
\begin{equation}\label{eq1856}
F(\xi)\geq F(\xi_1)+\langle\zeta_1, \xi-\xi_1\rangle +\frac{\mu}{2}\,|\xi-\xi_1|^2.
\end{equation}
Let \(\zeta'\in \partial F(\xi')\). By convexity of \(F\), we also have
\begin{equation}\label{eq1860}
F(\xi_1)\geq F(\xi')+\langle\zeta', \xi_1-\xi'\rangle,
\end{equation}
and
\[
\langle \zeta_1-\zeta', \xi_1-\xi'\rangle \geq 0,
\]
since the subdifferential of $F$ is a monotone map.
The latter inequality implies that
\[
\langle \zeta_1-\zeta', \xi-\xi_1\rangle \geq 0.
\]
Together with \eqref{eq1856} and \eqref{eq1860}, this yields
\[
\begin{split}
F(\xi) &\geq F(\xi') +\langle \zeta', \xi-\xi'\rangle +\frac{\mu}{2}\,|\xi-\xi_1|^2\\
&=F(\xi')+\langle \zeta', \xi-\xi' \rangle +\frac{\mu}{2}\,\mathcal{H}^1([\xi,\xi']\setminus B_R)^2,
\end{split}
\]
which settles the first case.
\par
In the second case, in addition to \eqref{eq1856}, we also use the fact that for every \(\zeta_2 \in \partial F(\xi_2)\),
\[
F(\xi_1)\geq F(\xi_2) +\langle \zeta_2, \xi_1-\xi_2\rangle,
\]
and for every \(\zeta'\in \partial F(\xi')\)
\[
F(\xi_2) \geq F(\xi') + \langle \zeta', \xi_2- \xi'\rangle +\frac{\mu}{2}\,|\xi_2-\xi'|^2,
\]
again by \eqref{equc_bis}.
By using as in the first case that \(\langle \zeta_1-\zeta', \xi-\xi_1\rangle \geq 0\) and  \(\langle \zeta_2-\zeta', \xi_1-\xi_2\rangle \geq 0\), we thus obtain
\[
\begin{split}
F(\xi) &\geq F(\xi') +\langle \zeta', \xi-\xi'\rangle +\frac{\mu}{2}\,(|\xi-\xi_1|^2+|\xi'-\xi_2|^2)\\
&\geq  F(\xi') +\langle \zeta', \xi-\xi'\rangle +\frac{\mu}{4}\,\mathcal{H}^1\left([\xi, \xi']\setminus B_R\right)^2.
\end{split}
\]
\end{itemize}
This completes the proof of \eqref{eq1739}.
\end{proof}
Thanks to the previous result, we can detail some consequences of the uniform convexity that we used in the proof of the Main Theorem.
\begin{lm}
\label{lemma_uc}
Let \(F\) be a convex function which is  $\mu-$uniformly convex outside a ball \(B_R=B_R(0)\subset \mathbb{R}^N\). Then we have:
\begin{enumerate}
\item[{\it i)}] for every \(\xi, \xi'\in  \mathbb{R}^N\),  and every \(\zeta\in \partial F(\xi)\), if \(|\xi|\geq 2R\) or \(|\xi'|\geq 2R\), then there holds
\begin{equation}\label{eq1838}
F(\xi')\geq F(\xi) +\langle \zeta, \xi'-\xi\rangle +\frac{\mu}{36}\,|\xi'-\xi|^2;
\end{equation}
\item[{\it ii)}] for every \(\xi, \xi'\in  \mathbb{R}^N\), and every \(\zeta\in \partial F(\xi), \zeta'\in\partial F(\xi')\), if \(|\xi|\geq 2R\) or \(|\xi'|\geq 2R\) we have
\begin{equation}
\label{eq1838bis}
\langle \zeta-\zeta', \xi-\xi'\rangle \geq \frac{\mu}{18}\,|\xi-\xi'|^2;
\end{equation}
\item[{\it iii)}] for every \(\xi, \xi'\in \mathbb{R}^N\setminus B_{2\,R}\) and for every \(\theta \in [0,1]\),
\begin{equation}\label{eq1834}
F(\theta\, \xi + (1-\theta)\,\xi')\leq \theta\, F(\xi) + (1-\theta)\,F(\xi') -\frac{\mu}{36}\,\theta\,(1-\theta)\, |\xi-\xi'|^2;
\end{equation}
\item[{\it iv)}] for every \(\xi\in \mathbb{R}^N\),
\begin{equation}
\label{eq1830}
F(\xi)\geq \frac{\mu}{72}\,|\xi|^2 -\left(|F(0)|+\frac{18}{\mu}\,|\partial^o F(0)|^2\right),
\end{equation}
where $\partial^o F(0)$ denotes the element of $\partial F(0)$ having minimal Euclidean norm. 
\end{enumerate}
\end{lm}
\begin{proof}
We claim that  for every \(\xi \in \mathbb{R}^N\) and every \(\xi'\in \mathbb{R}^N\setminus B_{2R}\),
\begin{equation}
\label{eq1891}
\mathcal{H}^1\left([\xi, \xi']\setminus B_R\right) \geq \frac{1}{3}\, |\xi-\xi'|.
\end{equation}
Let us fix $\xi'\in\mathbb{R}^N\setminus B_{2\,R}$ and let $\xi\in\mathbb{R}^N$. We set $r=|\xi'-\xi|$. Then we see that among all $\xi$ such that $|\xi-\xi'|=r$, the length of the set \([\xi, \xi']\setminus B_R\)  is minimal for the vector 
\[
\xi'':=t\,\xi',\qquad \mbox{ with } t<1 \mbox{ such that } (1-t)=\frac{r}{|\xi'|}.
\]
Then we have
\[
\mathcal{H}^1([\xi',\xi'']\setminus B_R)=\left\{\begin{array}{ll}
r,& \mbox{ if } r\le |\xi'|-R,\\
|\xi'|-R& \mbox{ if } |\xi'|-R<r<|\xi'|+R,\\
r-2\,R, & \mbox{ if } |\xi'|+R\le r.
\end{array}
\right.
\]
Observe that
\[
|\xi'|-R\ge \frac{|\xi'|}{2},
\]
and for $|\xi'|-R<r<|\xi'|+R$ we have 
\[
r<\frac{3}{2}\,|\xi'|\le 3\,(|\xi'|-R)=3\,\mathcal{H}^1([\xi',\xi'']\setminus B_R).
\]
Similarly, for $|\xi'|+R\le r$, we have
\[
\mathcal{H}^1([\xi',\xi'']\setminus B_R)=r-2\,R=\frac{r}{3}+\left(\frac{2}{3}\,r-2\,R\right)\ge\frac{r}{3}+\left(\frac{2}{3}\,|\xi'|-\frac{4}{3}\,R\right)\ge\frac{r}{3}.
\]
By recalling that $r=|\xi-\xi'|$, the claim \eqref{eq1891} follows.
The inequality \eqref{eq1838} can now be easily deduced from \eqref{eq1739} and \eqref{eq1891}. 
\vskip.2cm\noindent
Inequality \eqref{eq1838bis} can be easily obtained from \eqref{eq1838}. Indeed, by exchanging the role of $\xi$ and $\xi'$ in \eqref{eq1838} we get
\[
F(\xi')-F(\xi)\ge\langle \zeta,\xi'-\xi\rangle+\frac{\mu}{36}\, |\xi-\xi'|^2,
\]
and
\[
F(\xi)-F(\xi')\ge\langle \zeta',\xi-\xi'\rangle+\frac{\mu}{36}\, |\xi-\xi'|^2.
\]
By combining these two inequalities we get \eqref{eq1838bis}.
\vskip.2cm\noindent
Let us take \(\xi, \xi'\in \mathbb{R}^N\setminus B_{2R}\). For every \(\theta\in [0,1]\) and every \(\zeta\in \partial F(\theta\, \xi+(1-\theta)\,\xi')\) by \eqref{eq1838} we get
 \[
F(\xi)\geq F(\theta\,\xi + (1-\theta)\,\xi') +(1-\theta)\,\langle \zeta, \xi-\xi'\rangle +\frac{\mu}{36}\,(1-\theta)^2|\xi'-\xi|^2,
\]
and
\[
F(\xi')\geq F(\theta\,\xi + (1-\theta)\,\xi') +\theta\, \langle \zeta, \xi'-\xi\rangle +\frac{\mu}{36}\, \theta^2\,|\xi'-\xi|^2.
\]
Then \eqref{eq1834} can be obtained by multipliying the first inequality by $\theta$, the second one by $(1-\theta)$ and then summing up.
\vskip.2cm\noindent
Finally, we use \eqref{eq1838} with \(\xi'\in \mathbb{R}^N\setminus B_{2R}\), $\xi=0$ and $\zeta=\partial^o F(0)$ as in the statement, then we obtain
\[
\begin{split}
F(\xi') \geq F(0) + \langle \zeta, \xi'\rangle + 
\frac{\mu}{36}\,|\xi'|^2 &\geq \frac{\mu}{72}\, |\xi'|^2-\left(|F(0)|+\frac{18}{\mu}\,\,|\zeta|^2\right).
\end{split}
\]
where we used Young inequality in the last passage. This proves \eqref{eq1830}.
\end{proof}

\subsection{Approximation issues}

This section is devoted to prove some approximation results we used in the proof of the Main Theorem.
\begin{lm}
\label{lemma_Fk}
Let \(F:\mathbb{R}^N\to\mathbb{R}\) be a convex function, which is \(\mu-\)uniformly convex outside the ball \(B_R\). Then there exists a nondecreasing sequence \(\{F_k\}_{k\in \mathbb{N}}\) of smooth convex functions which converges to \(F\) uniformly on bounded sets. Moreover, for every \(k\geq 2R\), \(F_k\) is \((\mu/36)-\)uniformly convex outside the ball \(B_{R+1}\).
\end{lm}
\begin{proof}
Let us set for simplicity $\mu'=\mu/36$.
For every $k\in\mathbb{N}$, we define at first
\[
\widetilde{F}_k(x):=\sup_{\substack{|y|\leq k\\ \zeta\in \partial F(y)}} \left[F(y)+\langle \zeta, x-y\rangle+ \frac{\mu'}{2}\,|x-y|^2\, 1_{\mathbb{R}^N\setminus B_{2\,R}}(y)\right],\qquad x\in\mathbb{R}^N.
\]
Of course, this is a nondecreasing sequence of convex functions.
If $k\le 2\,R$, then by convexity of $F$ for every $|y|\le k$, every $\zeta\in\partial F(y)$ and every $x\in\mathbb{R}^N$ we get
\begin{equation}
\label{pulmino}
F(y)+\langle \zeta, x-y\rangle+ \frac{\mu'}{2}\,|x-y|^2\, 1_{\mathbb{R}^N\setminus B_{2\,R}}(y)\le F(x).
\end{equation}
If $k>2\, R$ and $|y|\le k$, we have two possibilities: either $|y|\le 2\,R$ or $|y|>2\,R$. In the first case we still have \eqref{pulmino} for every $\zeta\in \partial F(y)$ and $x\in\mathbb{R}^N$, simply by convexity of $F$. In the second case, we can appeal to Lemma \ref{lemma_uc}: indeed, for every \(x \in \mathbb{R}^N\) and every \(\zeta\in \partial F(y)\), we have
\[
F(x)\geq F(y)+\langle \zeta, x-y\rangle+ \mu'\,|x-y|^2.
\]
In any case, we obtain that for every \(x\in \mathbb{R}^N\) 
\[
\widetilde{F}_k(x)\leq F(x),
\] 
and the equality holds when \(x\in \overline{B_k}\). In particular, for every \(k\geq R\) the function \(\widetilde{F}_k\) is \(\mu-\)uniformly convex on \(\overline{B_k} \setminus B_R\). 
\par
When \(k\geq 2\,R\) and \(|x|\geq 2\,R\), we claim that
\begin{equation}
\label{eq1813}
\widetilde{F}_k(x)=\sup_{\substack{2\,R\leq |y|\leq k\\ \zeta\in \partial F(y)}} \left[F(y)+\langle \zeta, x-y\rangle+ \frac{\mu'}{2}\,|x-y|^2\right].
\end{equation}
This follows from the fact that for every \(y_0 \in B_{2R}\) and \(\zeta_0\in \partial F(y_0)\), there exists \(y\in \overline{B_{k}}\setminus B_{2R}\) and \(\zeta\in \partial F(y)\) such that  
\begin{equation}\label{eq1817}
F(y)+\langle \zeta, x-y\rangle\geq F(y_0)+\langle \zeta_0, x-y_0\rangle.
\end{equation}
Indeed, take any \(y\in [y_0,x]\cap(\overline{B_k}\setminus B_{2R})\). Then, by convexity of \(F\),
\[
F(y)\geq F(y_0) + \langle \zeta_0, y-y_0\rangle.
\]
Hence, by using this and the fact that \(y-x=t\,(y_0-y)\) for some \(t\geq 0\), we can infer
\[
\begin{split}
F(y)+\langle \zeta, x-y\rangle &\geq F(y_0)+\langle \zeta_0, y-y_0\rangle+\langle \zeta, x-y\rangle\\
&=F(y_0)+\langle \zeta_0, x-y_0\rangle+\langle \zeta_0-\zeta, y-x\rangle\\
&=F(y_0)+\langle \zeta_0, x-y_0\rangle+t\,\langle \zeta_0-\zeta, y_0-y\rangle\\
&\geq F(y_0)+\langle \zeta_0, x-y_0\rangle.
\end{split}
\]
In the last line, we have used 
\[
\langle \zeta_0-\zeta,y_0-y\rangle\ge 0,
\]
which follows from the convexity of \(F\), by recalling that $\zeta\in\partial F(y)$ and $\zeta_0\in\partial F(y_0)$.  This proves \eqref{eq1817} and thus \eqref{eq1813}.
\par
It follows that \(\widetilde{F}_k\) is \(\mu'-\)uniformly convex on \(\mathbb{R}^N\setminus B_{2\,R}\) as the supremum of \(\mu'-\)uniformly convex functions on \(\mathbb{R}^N\setminus B_{2\,R}\). Since $\mu'<\mu$, on the whole we get that \(\widetilde{F}_k\) is \(\mu'-\)uniformly convex on \(\mathbb{R}^N\setminus B_R\).
\vskip.2cm\noindent
In the remaining part of the proof, we fix some \(k\geq 2\,R\).
We claim that  for every  \(x\in \mathbb{R}^N\),
\begin{equation}
\label{eqWaterloo}
\widetilde{F}_{k+1}(x)\geq \widetilde{F}_{k}(x) + \mu'\,(|x|-k-1)_+.
\end{equation}
If $|x|\le k+1$ this is immediate, thus let us assume that \(|x|>k+1\). Let \(y_0\in \overline{B_k}\) and \(\zeta_0\in \partial F(y_0)\) achieving the supremum in the definition of $\widetilde F_k$, i.e. 
\begin{equation}\label{eqbrr}
\widetilde{F}_k(x) = F(y_0) + \langle \zeta_0, x-y_0 \rangle  + \frac{\mu'}{2}|x-y_0|^2\,1_{\mathbb{R}^N\setminus B_{2\,R}}(y_0).
\end{equation}
Let \(y\in \partial B_{k+1}\cap [x,y_0]\) be such that \([y,x]\cap B_R = \emptyset\).
\begin{figure}[h]
\includegraphics[scale=.35]{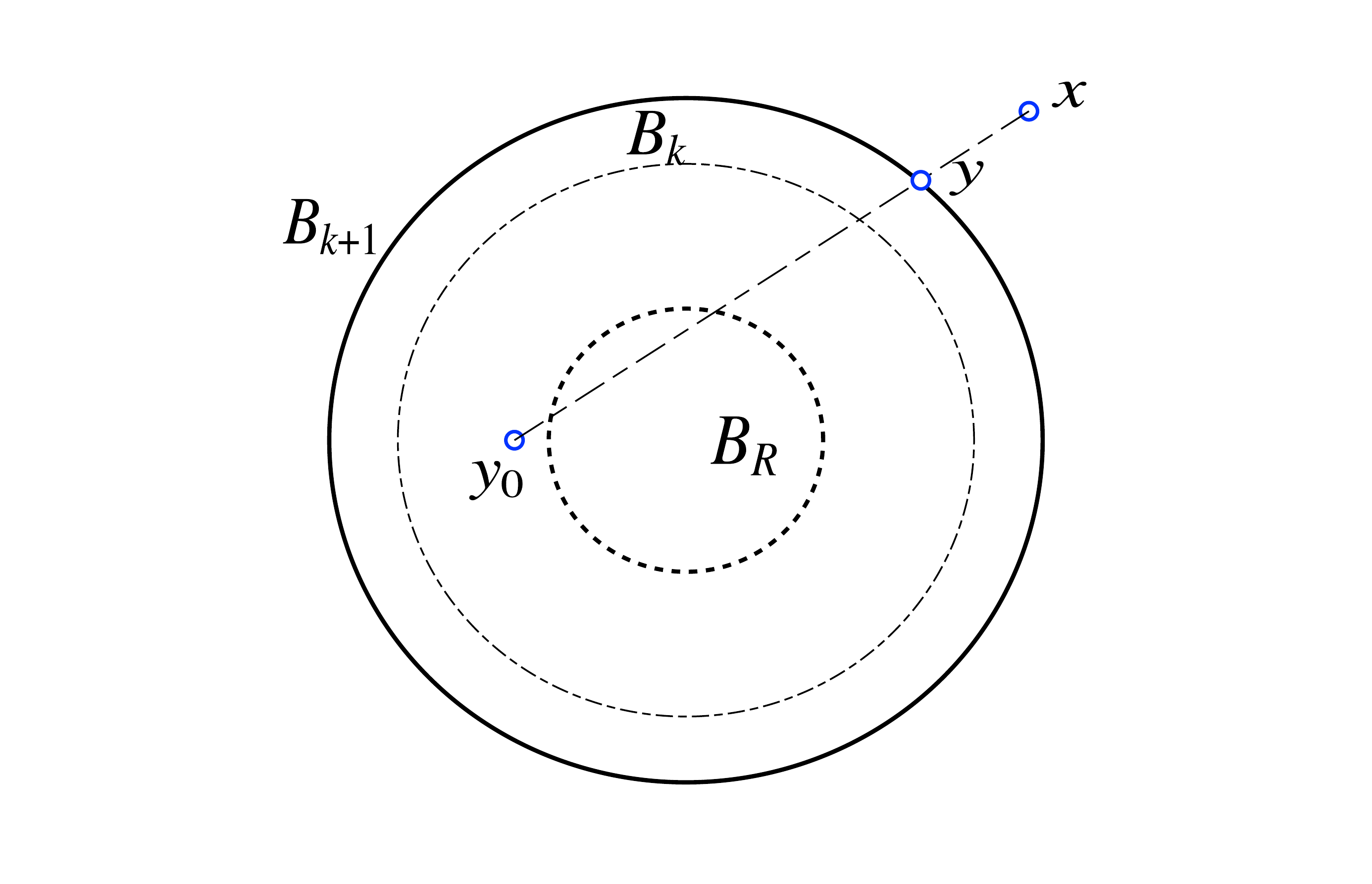}
\caption{The construction for the proof of \eqref{eqWaterloo}.}
\end{figure} 
Then by definition of \(\widetilde{F}_{k+1}\) and Lemma \ref{lemma_uc}, we get for every \(\zeta \in \partial F(y)\) and every \(\zeta_0 \in \partial F(y_0)\)
\[
\widetilde{F}_{k+1}(x)\geq F(y)+\langle \zeta, x-y\rangle + \frac{\mu'}{2}|x-y|^2,
\]
and
\[
F(y)\geq F(y_0) + \langle \zeta_0, y-y_0\rangle +\mu'\,|y-y_0|^2.
\]
Combining these two inequalities, we get
\begin{equation}
\label{eq1829}
\widetilde{F}_{k+1}(x)\geq F(y_0) + \langle \zeta_0, x-y_0\rangle + \frac{\mu'}{2}\,|x-y_0|^2 + A, 
\end{equation}
where
\begin{equation}
\label{eqA}
A= \langle \zeta_0-\zeta, y-x\rangle + \frac{\mu'}{2}\,|x-y|^2 +\mu'\,|y-y_0|^2 - \frac{\mu'}{2}\,|x-y_0|^2.
\end{equation}
If we now use \eqref{eq1838bis}, we obtain
\[
\langle \zeta_0-\zeta, y_0-y\rangle \geq 2\,\mu'\,|y_0-y|^2.
\]
Since \(y-x=(y_0-y)\,|y-x|/|y_0-y|\), this implies
\[
\langle \zeta_0-\zeta, y-x\rangle \geq 2\,\mu'\,|y_0-y|\,|y-x|.
\]
By inserting this into \eqref{eqA} and observing that \(|x-y_0|=|x-y|+|y-y_0|\), we obtain
\[
A\geq \mu'\,|y_0-y|\,|y-x|+\frac{\mu'}{2}\,|y-y_0|^2.
\]
In view of \eqref{eq1829}, \eqref{eqbrr}  and also using the fact that 
\[
|y_0-y|\geq 1 \qquad \mbox{ and }\qquad |y-x|\geq |x|-|y|\geq |x|-(k+1),
\] 
this finally implies 
\[
\widetilde{F}_{k+1}(x)\geq \widetilde{F}_k(x)+\mu'\,(|x|-k-1),
\]
and \eqref{eqWaterloo} is proved.
\vskip.2cm\noindent
We now establish the following Lipschitz estimate for $\widetilde F_k$: for every \(x, x'\in \mathbb{R}^N\),
\begin{equation}
\label{eqAusterlitz}
\left|\widetilde{F}_k(x)-\widetilde{F}_k(x')\right| \le \left(L_k + \frac{\mu'}{2}(|x|+|x'|+2\,k)\right)|x-x'|,
\end{equation}
where  \(L_k\) is the Lipschitz constant of \(F\) on \(B_k\). 
Indeed, for every \(y\in \overline{B_k}\) and every \(\zeta \in \partial F(y)\), we have\footnote{We use the following elementary manipulations
\[
|x-y|^2=|x'-y|^2+(|x|^2-|x'|^2)+2\,\langle x'-x,y\rangle,
\]
and 
\[
\begin{split}
(|x|^2-|x'|^2)+2\,\langle x'-x,y\rangle&\le (|x|-|x'|)\,(|x|+|x'|)+2\, |x-x'|\,|y|\\
&\le |x-x'|\,(|x|+|x'|+2\,|y|).
\end{split}
\]}
\[
\begin{split}
F(y) +\langle \zeta, x-y\rangle + \frac{\mu'}{2}\,|x-y|^2\, 1_{\mathbb{R}^N\setminus B_{2R}}(y) &\leq F(y) + \langle \zeta, x'-y\rangle + \frac{\mu'}{2}\,|x'-y|^2\, 1_{\mathbb{R}^N\setminus B_{2R}}(y) \\ & \quad + |\zeta|\,|x-x'| +  \frac{\mu'}{2}\,|x-x'|\,\left(|x|+|x'|+2\,|y|\right)\\
& \leq \widetilde{F}_k(x') + \left(L_k + \frac{\mu'}{2}(|x|+|x'|+2\,k)\right)|x-x'|.
\end{split}
\]
By exchanging the role of $x$ and $x'$ we get \eqref{eqAusterlitz}.
\par
Let us introduce a family \(\{\rho_{\varepsilon}\}_{\varepsilon>0} \subset C^{\infty}_0(B_{\varepsilon})\) of smooth mollifiers. For some sequence \(\{\varepsilon_k\}_{k\in \mathbb{N}}\subset (0,1/2)\) to be specified later, let us consider
\[
F_k(x):=\widetilde{F}_k*\rho_{\varepsilon_k}-\frac{1}{k}.
\]
By Lemma \ref{lmunifconvsndder}, every \(F_k\) is a smooth \(\mu'-\)uniformly convex function outside \(B_{R+1}\) and the sequence \(\{F_k\}_{k\in \mathbb{N}}\)  uniformly converges on bounded sets to \(F\). It remains to prove that \(\{F_k\}_{k\in \mathbb{N}}\) is nondecreasing.
By \eqref{eqAusterlitz}, for every \(x\in \mathbb{R}^N\),
\[
F_k(x)\leq \widetilde{F}_k(x) + \left(L_k + \frac{\mu'}{2}(2\,|x|+\varepsilon_k+2\,k)\right)\varepsilon_k-\frac{1}{k}.
\]
Moreover, by \eqref{eqWaterloo} we have
\[
\widetilde{F}_k(x)\leq \widetilde{F}_{k+1}(x)-\mu'\,(|x|-k-1)_{+}.
\]
 Since \(\widetilde{F}_{k+1}\) is convex, by Jensen inequality we also have 
\[
\widetilde{F}_{k+1}(x)\leq \widetilde{F}_{k+1}*\rho_{\varepsilon_{k+1}}(x).
\] 
Hence in order to have \(F_k(x)\leq F_{k+1}(x)\) it is sufficient that
for every \(x\in \mathbb{R}^N\)
\begin{equation}
\label{eqBelleepoque}
\left(L_k + \frac{\mu'}{2}(2\,|x|+\varepsilon_k+2k)\right)\varepsilon_k-\frac{1}{k} \leq \mu'\,(|x|-k-1)_{+} -\frac{1}{k+1}.
\end{equation}
When \(|x|\leq 2\,(k+1)\), by recalling that $\varepsilon_k<1/2$ we have
\[
\left(L_k + \frac{\mu'}{2}(2\,|x|+\varepsilon_k+2\,k)\right)\varepsilon_k-\frac{1}{k}\le \left(L_k + \frac{\mu'}{2}(6\,k+5)\right)\varepsilon_k-\frac{1}{k},
\]
while for the right-hand side of \eqref{eqBelleepoque} 
\[
\mu'\,(|x|-k-1)_{+} -\frac{1}{k+1}\ge -\frac{1}{k+1}.
\]
Hence \eqref{eqBelleepoque} holds true provided
\begin{equation}
\label{choiceeps1}
\varepsilon_k\leq \Gamma^-_k:=\left(\frac{1}{k}-\frac{1}{k+1}\right)\,\frac{1}{L_k +(6\,k+5)\,\mu'/2}.
\end{equation}
When \(|x|>2\,(k+1)\), the left-hand side of \eqref{eqBelleepoque} can be estimated by
\[
\left(L_k + \frac{\mu'}{2}(2\,|x|+\varepsilon_k+2\,k)\right)\varepsilon_k-\frac{1}{k}\le \varepsilon_k\,\mu'\,|x|+\varepsilon_k\left(L_k+\mu'\,(k+1)\right)-\frac{1}{k}
\]
while for the right-hand side we have 
\[
\mu'\,(|x|-k-1)_{+} -\frac{1}{k+1}\ge \frac{\mu'}{2}|x|-\frac{1}{k+1}.
\] 
In that case, by recalling that $\varepsilon_k<1/2$ we only need to take
\begin{equation}
\label{choiceeps2}
\varepsilon_k\leq \Gamma^+_k:=\left(\frac{1}{k}-\frac{1}{k+1}\right)\frac{1}{L_k +\mu'\,(k+1)}.
\end{equation}
Observe that both $\{\Gamma^+_k\}_{k\in\mathbb{N}}$ and $\{\Gamma^-_k\}_{k\in\mathbb{N}}$ converge to $0$ as $k$ goes to $\infty$. Hence, by choosing the decreasing sequence 
\[
\varepsilon_1=\min\left\{\Gamma^-_1,\,\Gamma^+_1,\,\frac{1}{2}\right\},\qquad \varepsilon_{k+1}=\min\{\Gamma^-_k,\,\Gamma^+_k,\,\varepsilon_k\},\qquad k\in\mathbb{N},
\]
this satisfies both \eqref{choiceeps1} and \eqref{choiceeps2} and thus \(\{F_k\}_{k\in \mathbb{N}}\) satisfies all the required properties.
\end{proof}

\begin{lm}
\label{lm:pierresque}
Let $F:\mathbb{R}^N\to\mathbb{R}$ be a convex function, which is $\Phi-$uniformly convex outside the ball $B_R$. Then for every $Q>R$ there exists a convex function $F_Q:\mathbb{R}^N\to\mathbb{R}$ with the following properties:
\begin{enumerate}
\item[{\it i)}] $F_Q\equiv F$ in $B_Q$;
\vskip.2cm
\item[{\it ii)}] $F$ is $\mu_Q-$uniformly convex outside the ball $B_R$, where\footnote{Observe that $\mu_Q>0$ thanks to fact that $\Phi$ is continuous and $\Phi(t)>0$ for $t>0$.} 
\begin{equation}
\label{defeqmuQ}
\mu_Q=\min\left\{1,\min_{t\in [2\,R, 4\,Q]}\Phi(t)\right\}.
\end{equation}
\end{enumerate}
\end{lm}
\begin{proof}
For every $Q>R$, we define the function $F_Q:\mathbb{R}^N\to\mathbb{R}$ by
\[
F_Q(x)=F(x)+\mu_Q\,J_Q(x),\qquad \mbox{ where }J_Q(x):=(|x|-Q)_+^2.
\]
Of course, this is a convex function such that $F_Q\equiv F$ in $B_Q$. 
\vskip.2cm\noindent
In order to verify property {\it ii)}, we first observe that $\mu_Q\,J_Q$ is $\mu_Q-$uniformly convex outside $B_{2\,Q}$ (see Lemma \ref{lm:coglione} below). 
We consider again a sequence \(\{\rho_{\varepsilon}\}_{\varepsilon>0}\subset C^{\infty}_0(B_{\varepsilon})\) of standard mollifiers. Then for every \(\eta \in \mathbb{R}^N\) and every \(\xi \in \mathbb{R}^{N}\setminus B_{R+\varepsilon}\), we have
\[
\langle D^2(F_Q*\rho_{\varepsilon})(\xi)\,\eta, \eta \rangle = \langle D^2(F*\rho_{\varepsilon})(\xi)\,\eta, \eta \rangle + \langle D^2 (J_Q*\rho_{\varepsilon}(\xi))\,\eta, \eta \rangle. 
\]
Since  $F$ is $\Phi-$uniformly convex outside $B_R$, the first part of Lemma \ref{lmunifconvsndder} implies that when \(\xi \in B_{2\,Q+\varepsilon}\setminus  B_{R+\varepsilon}\),
\[
\langle D^2(F*\rho_{\varepsilon})(\xi)\eta, \eta \rangle \geq \left(\min_{t\in [2\,R, 4(Q+\varepsilon)]}\Phi(t)\right)\,|\eta|^2.
\]
Since $\mu_Q\,J_Q$ is $\mu_Q-$uniformly convex outside $B_{2\,Q}$, we get similarly when \(\xi \in \mathbb{R}^N\setminus  B_{2\,Q+\varepsilon}\),
\[
\langle D^2(\mu_QJ_Q*\rho_{\varepsilon})(\xi)\,\eta, \eta \rangle \geq \mu_Q\,|\eta|^2.
\]
In any case, we thus have for every $\xi\in \mathbb{R}^N\setminus B_{R+\varepsilon}$
\[
\langle D^2(F_Q*\rho_{\varepsilon})(\xi)\,\eta, \eta \rangle \geq \mu_Q\,|\eta|^2.
\]
By the second part of Lemma \ref{lmunifconvsndder}, this proves that  \(F_Q\) is \(\mu_Q\) uniformly convex outside \(B_R\).
\end{proof}

\begin{lm}[A useful function]
\label{lm:coglione}
The function
\[
{J}_Q(x)=(|x|-Q)_+^2,
\]
is $1-$uniformly convex outside the ball $B_{2\,Q}$. 
\end{lm}
\begin{proof}
We first observe that $J_Q$ is $C^2$ outside $B_{Q}$. Thus it is sufficient to compute the Hessian of $J_Q$ in $\mathbb{R}^N\setminus B_{2\,Q}$.
Let $x\in \mathbb{R}^N\setminus B_{2\,Q}$, we have
\[
D^2 J_Q(x)=2\,\frac{|x|-Q}{|x|}\, \mathrm{Id}_N+2\,\frac{x\otimes x}{|x|^2}-2\,(|x|-Q)\,\frac{x\otimes x}{|x|^3}.
\]
For every $\eta\in \mathbb{R}^N$ we get
\[
\langle D^2 J_Q(x)\,\eta,\eta\rangle=2\,\frac{|x|-Q}{|x|}\,|\eta|^2+2\,\left[1-\frac{|x|-Q}{|x|}\right]\,\left(\frac{\langle x,\eta\rangle}{|x|}\right)^2\ge 2\,\left(1-\frac{Q}{|x|}\right)\, |\eta|^2,
\]
and thus the conclusion follows, by using that $|x|\ge 2\,Q$.
\end{proof}

\end{document}